\pdfoutput=1

\documentclass{siamart0216}

\usepackage{amsmath,amssymb,amsfonts}
\newtheorem*{remark}{Remark}

\usepackage[titletoc,toc,title]{appendix}
\usepackage{listings}
\usepackage{array} 
\usepackage{mathtools}
\usepackage{bm}
\usepackage{bbm}

\usepackage{tikz}
\usepackage[normalem]{ulem}
\usepackage{hhline}

\usepackage{graphicx}
\usepackage{subfig}
\usepackage{color}

\usepackage{pgfplots}
\usepackage{pgfplotstable}
\definecolor{markercolor}{RGB}{124.9, 255, 160.65}
\pgfplotsset{
compat=1.3,
width=10cm,
tick label style={font=\small},
label style={font=\small},
legend style={font=\small}
}

\usetikzlibrary{calc}
\usetikzlibrary{intersections} 

\usepackage{stmaryrd}

\renewcommand{\tilde}{\widetilde}
\renewcommand{\hat}{\widehat}

\newcommand{\td}[2]{\frac{{\rm d}#1}{{\rm d}{\rm #2}}}
\newcommand{\pd}[2]{\frac{\partial#1}{\partial#2}}
\newcommand{\nor}[1]{\left\| #1 \right\|}
\newcommand{\LRp}[1]{\left( #1 \right)}
\newcommand{\LRs}[1]{\left[ #1 \right]}

\newcommand{\LRb}[1]{\left| #1 \right|}
\newcommand{\LRc}[1]{\left\{ #1 \right\}}

\newcommand{\LRl}[1]{\left. #1 \right|}

\newcommand{\Grad} {\ensuremath{\nabla}}

\newcommand{\jump}[1] {\ensuremath{\llbracket#1\rrbracket}}
\newcommand{\avg}[1] {\ensuremath{\LRc{\!\{#1\}\!}}}

\newcommand{\eval}[2][\right]{\relax
  \ifx#1\right\relax \left.\fi#2#1\rvert}

\newcommand{\note}[1]{#1}

\newcommand{\diag}[1]{{\rm diag}\LRp{#1}}

\newcolumntype{C}[1]{>{\centering\let\newline\\\arraybackslash\hspace{0pt}}m{#1}}

\newcommand*\diff[1]{\mathop{}\!{\mathrm{d}#1}}

\makeatletter
\renewcommand\d[1]{\mspace{6mu}\mathrm{d}#1\@ifnextchar\d{\mspace{-3mu}}{}}
\makeatother

\newcommand\myeq{\mathrel{\stackrel{\makebox[0pt]{\mbox{\normalfont\tiny def}}}{=}}}

\author{Jesse Chan, David C.\ Del Rey Fernandez, Mark H.\ Carpenter}
\title{Efficient entropy stable Gauss collocation methods}
\graphicspath{{./figs/}}

\begin{document}

\maketitle

\begin{abstract}
The construction of high order entropy stable collocation schemes on quadrilateral and hexahedral elements has relied on the use of Gauss-Legendre-Lobatto collocation points \cite{fisher2013high, carpenter2014entropy, gassner2016split} and their equivalence with summation-by-parts (SBP) finite difference operators \cite{gassner2013skew}.  In this work, we show how to efficiently generalize the construction of semi-discretely entropy stable schemes on tensor product elements to Gauss points and generalized SBP operators.  Numerical experiments suggest that the use of Gauss points significantly improves accuracy on curved meshes.  
\end{abstract}

\section{Introduction}

Time dependent nonlinear conservation laws are ubiquitous in computational fluid dynamics, for which high order methods are increasingly of interest.  Such methods are more accurate per degree of freedom than low order methods, while also possessing much smaller numerical dispersion and dissipation errors.  This makes high order methods especially well suited to time-dependent simulations.  In this work, we focus specifically on discontinuous Galerkin methods on unstructured quadrilateral and hexahedral meshes.  These methods combine properties of high order approximations with the geometric flexibility of unstructured meshing.  

However, high order methods are notorious for being more prone to instability compared to low order methods \cite{wang2013high}.  This instability is addressed through various stabilization techniques (e.g.\ artificial viscosity, filtering, slope limiting).  However, these techniques often reduce accuracy to first or second order, and can prevent solvers from realizing the advantages of high order approximations.  Moreover, it is often not possible to prove that a high order scheme does not blow up even in the presence of stabilization.  This ambiguity can necessitate the re-tuning of stabilization parameters, as parameters which are both stable and accurate for one problem or discretization setting  may provide either too little or too much numerical dissipation for another.  

The instability of high order methods is rooted in the fact that discretizations of nonlinear conservation laws do not typically satisfy a discrete analogue of the conservation or dissipation of energy (entropy).  For low order methods, the lack of discrete stability can be offset by the presence of numerical dissipation, which serves as a stabilization mechanism.  Because high order methods possess low numerical dissipation, the absence of a discrete stability property becomes more noticeable, manifesting itself through increased sensitivity and instability.  

The dissipation of entropy serves as an energetic principle for nonlinear conservation laws \cite{dafermos2005compensated}, and requires the use of the chain rule in its proof.  Discrete instability is typically tied to the fact that, when discretizing systems of nonlinear PDEs, the chain rule does not typically hold at the discrete level.  The lack of a chain rule was circumvented by using a non-standard ``flux differencing'' formulation \cite{fisher2013high, carpenter2014entropy, gassner2016split, gassner2017br1}, which is key to constructing semi-discretely entropy stable high order schemes on unstructured quadrilateral and hexahedral meshes.  Flux differencing replaces the derivative of the nonlinear flux with the discrete differentiation of an auxiliary quantity.  This auxiliary quantity is computed through the evaluation of a two-point entropy conservative flux \cite{tadmor1987numerical} using pairs of solution values at quadrature points.  These entropy stable schemes were later extended to non-tensor product elements using GLL-like quadrature points on triangles and tetrahedra \cite{chen2017entropy, crean2018entropy}.  More recently, the construction of efficient entropy stable schemes was extended to more arbitrary choices of basis and quadrature \cite{chan2017discretely, chan2018discretely}.  

While entropy stable collocation schemes have been constructed on Gauss-like quadrature points without boundary nodes \cite{crean2017high}, the inter-element coupling terms for such schemes introduce an ``all-to-all'' coupling between degrees of freedom on two neighboring elements in one dimension (on tensor product elements in higher dimensions, these coupling terms couple together lines of nodes).  These coupling terms require evaluating two-point fluxes between solution states at all collocation nodes on two neighboring elements, resulting in significantly more communication and computation compared to collocation schemes based on point sets containing boundary nodes.  This work introduces efficient and entropy stable inter-element coupling terms for Gauss collocation schemes which require only communication of face values between neighboring elements.  The construction of these terms follows the framework introduced in \cite{chan2017discretely, chan2018discretely} for triangles and tetrahedra.  

The main motivation for exploring tensor product (quadrilateral and hexahedral) elements is the significant reduction in the number of operations required compared to high order entropy stable schemes on simplicial meshes \cite{chan2017discretely, chan2018discretely}.  Entropy stability schemes on simplicial elements require evaluating two-point fluxes between solution states at all quadrature points on an element.  For a degree $N$ approximation, the number of quadrature points on a simplex scales as $O(N^d)$ in $d$ dimensions, and results in $O(N^{2d})$ two-point flux evaluations per element.  In contrast, entropy stable schemes on quadrilateral and hexahedral elements require only the evaluation of two-point fluxes along lines of nodes due to a tensor product structure, resulting in $O(N^{d+1})$ evaluations in $d$ dimensions.  

In Section~\ref{sec:0}, we briefly review the derivation of continuous entropy inequalities for systems of nonlinear conservation laws.  In Section~\ref{sec:1}, we describe how to construct entropy stable discretizations of nonlinear conservation laws using different quadrature points on affine tensor product elements.  In Section~\ref{sec:2}, we describe how to extend this construction to curvilinear elements, and Section~\ref{sec:3} presents numerical results which confirm the high order accuracy and stability of the proposed method for smooth, discontinuous, and under-resolved (turbulent) solutions of the compressible Euler equations in two and three dimensions.  

\section{A brief review of entropy stability theory}
\label{sec:0}

We are interested in methods for the numerical solution of systems of conservation laws in $d$ dimensions
\begin{equation}
\pd{\bm{u}}{t} + \sum_{i=1}^d \pd{\bm{f}_i\LRp{\bm{u}}}{x_i} = 0,
\label{eq:nonlinpde}
\end{equation}
where $\bm{u}$ denotes the conservative variables, $\bm{f}_i(\bm{u})$ are nonlinear fluxes, and $x_i$ denotes the $i$th coordinate.  Many physically motivated conservation laws admit a statement of stability involving a convex scalar entropy $S(\bm{u})$.  We first define the entropy variables $\bm{v}(\bm{u})$ to be the gradient of the entropy $S(\bm{u})$ with respect to the conservative variables 
\[
\bm{v} \myeq \pd{S(\bm{u})}{\bm{u}}.  
\]
For a convex entropy, $\bm{v}(\bm{u})$ defines an invertible mapping from conservative to entropy variables.  We denote the inverse of this mapping (from entropy to conservative variables) by $\bm{u}(\bm{v})$.  

At the continuous level, it can be shown (for example, in \cite{dafermos2005compensated}) that vanishing viscosity solutions to (\ref{eq:nonlinpde}) satisfy the strong form of an entropy inequality 
\begin{gather}
\pd{S(\bm{u})}{t} + \sum_{i=1}^d \pd{F_i(\bm{u})}{x_i} \leq 0, \qquad F_i(\bm{u}) = \bm{v}^T\pd{\bm{f}_i}{\bm{u}}, 
\label{eq:entropyineq}
\end{gather}
where $F_i$ denotes the $i$th scalar entropy flux function.  Integrating (\ref{eq:entropyineq}) over a domain $\Omega$ and applying the divergence theorem yields an integrated entropy inequality
\begin{equation}
\int_{\Omega} \pd{S(\bm{u})}{t} + \int_{\partial \Omega} \sum_{i=1}^d \bm{n}_i \LRp{\bm{v}^T\bm{f}_i(\bm{u}) - \psi_i(\bm{u})} \leq 0,
\label{eq:weakentropyineq}
\end{equation}
where $\psi_i(\bm{u}) = \bm{v}^T\bm{f}_i(\bm{u}) - F_i(\bm{u})$ denotes the $i$th entropy potential, $\partial \Omega$ denotes the boundary of $\Omega$ and $\bm{n}_i$ denotes the $i$th component of the outward normal on $\partial \Omega$.  Roughly speaking, this implies that the time rate of change of entropy is less than or equal to the flux of entropy through the boundary.

\section{Entropy stable Gauss and Gauss-Legendre-Lobatto collocation methods}
\label{sec:1}

The focus of this paper is on entropy stable high order collocation methods which satisfy a semi-discrete version of the entropy inequality (\ref{eq:weakentropyineq}).  These methods collocate the solution at some choice of collocation nodes, and are applicable to tensor product meshes consisting of quadrilateral and hexahedral elements.  

Entropy stable collocation methods have largely utilized Gauss-Legendre-Lobatto (GLL) nodes \cite{fisher2013high, carpenter2014entropy, gassner2016split, gassner2017br1}, which contain points on the boundary.  The popularity of GLL nodes can be attributed in part to a connection made in \cite{gassner2013skew}, where it was shown by Gassner that collocation DG discretizations based on GLL nodes could be recast in terms of summation-by-parts (SBP) operators.  This equivalence allowed Gassner to leverage existing finite difference formulations to produce stable high order discretizations of the nonlinear Burgers' equation.  

GLL quadratures contain boundary points, which greatly simplifies the construction of inter-element coupling terms for entropy stable collocation schemes.  However, it is also known that the use of GLL quadrature within DG methods under-integrates the mass matrix, which can lead to solution ``aliasing'' and lower accuracy \cite{parsani2016entropy}.  In this work, we explore entropy stable collocation schemes based on Gauss quadrature points instead of GLL points.  

This comparison is motivated by the accuracy of each respective quadrature rule.  While $(N+1)$-point GLL quadrature rules are exact for polynomial integrands of degree $(2N-1)$, $(N+1)$-point Gauss quadrature rules are exact for polynomials of degree $(2N+1)$.  This higher accuracy of Gauss quadrature has been shown to translate to lower errors and slightly improved rates of convergence in simulations of wave propagation and fluid flow \cite{kopriva2010quadrature, hindenlang2012explicit, chan2015gpu}.  However, Gauss points have not been widely used to construct entropy stable discretizations due to the lack of efficient, stable, and high order accurate inter-element coupling terms, known as simultaneous approximation terms (SAT) in the finite difference literature \cite{fernandez2014review, crean2017high, fernandez2018simultaneous}.  SATs for Gauss points are non-compact, in the sense that they introduce all-to-all coupling between degrees of freedom on neighboring elements in one dimension.  This results in greater communication between elements, as well as a significantly larger number of two-point flux evaluations and floating point operations.  

It is possible to realize the improved accuracy of Gauss points while avoiding non-compact SATs through a staggered grid formulation, where the solution is stored at Gauss nodes but interpolated to a set of higher degree $(N+2)$ GLL ``flux'' points for computation \cite{parsani2016entropy}.  Because GLL nodes include boundary points, compact and high order accurate SAT terms can be constructed for the flux points.  After performing computations on the flux points, the results are interpolated back to Gauss points and used to evolve the solution forward in time.  Figure~\ref{fig:nodesets} shows an illustration of GLL, staggered grid, and Gauss point sets for a 2D quadrilateral element.  

\begin{figure}
\centering
\subfloat[GLL nodes]{\includegraphics[width=.3\textwidth]{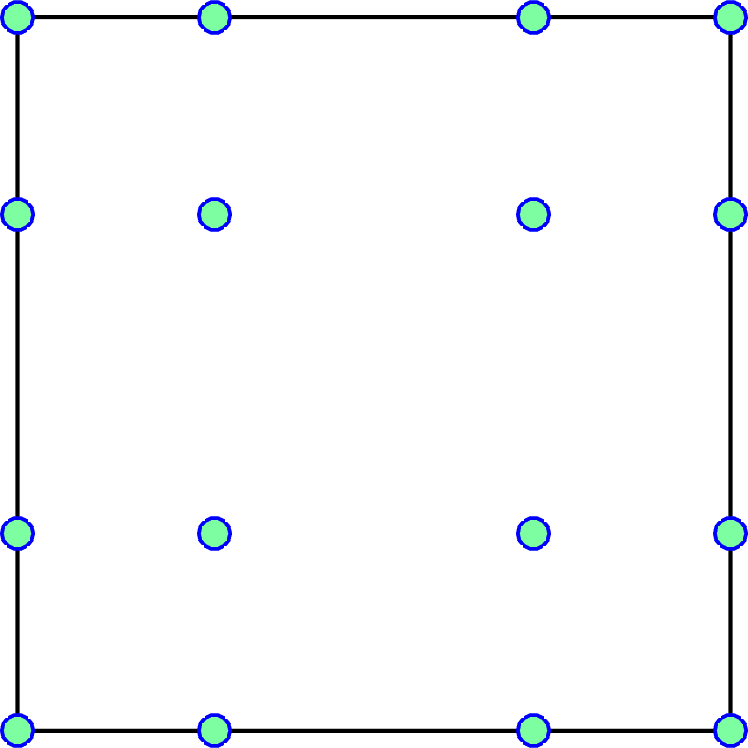}}
\hspace{.5em}
\subfloat[Staggered grid nodes]{\includegraphics[width=.3\textwidth]{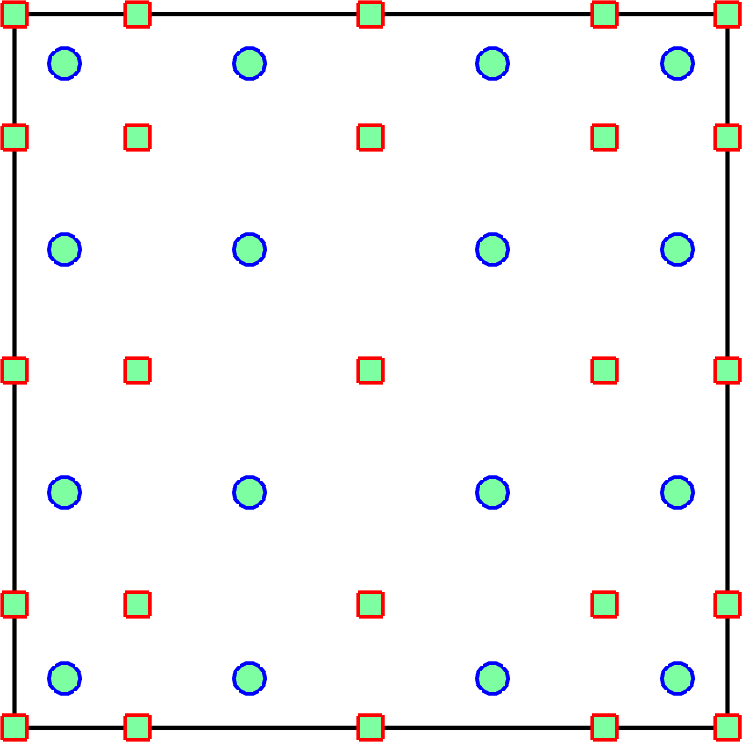}}
\hspace{.5em}
\subfloat[Gauss nodes]{\includegraphics[width=.3\textwidth]{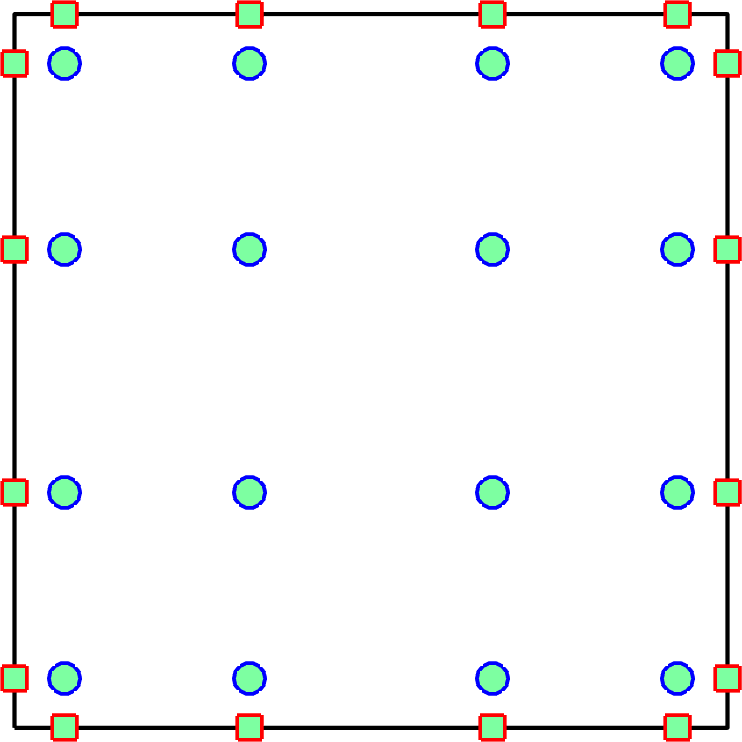}}
\caption{Examples of nodal sets under which efficient entropy stable schemes can be constructed.  This work focuses on the construction of efficient and accurate SAT terms for Gauss nodal sets.}
\label{fig:nodesets}
\end{figure}

The following sections will describe how to construct efficient high order entropy stable schemes using Gauss points.  These schemes are based on ``decoupled'' SBP operators introduced in \cite{chan2017discretely, chan2018discretely}, which are applicable to general choices of basis and quadrature.  By choosing a tensor product Lagrange polynomial basis and $(N+1)$ point Gauss quadrature rules, we recover a Gauss collocation scheme.  The high order accuracy and entropy stability of this scheme are direct results of theorems presented in  \cite{chan2017discretely, chan2018discretely}.  However, we will also present a different proof of entropy stability in one dimension for completeness.  

\subsection{Gauss nodes and generalized summation by parts operators}
\label{sec:gsbp}
We assume the solution is collocated at $(N+1)$ quadrature points $x_i$ with associated quadrature weights $w_i$.  We do not make any assumptions on the points, in order to accommodate both GLL and Gauss nodes using this notation.  The collocation assumption is equivalent to approximating the solution using a degree $N$ Lagrange basis $\ell_j(x)$ defined over the $(N+1)$ quadrature points.  

Let $\bm{D}$ denote the nodal differentiation matrix, and let $\bm{V}_f$ denote the \note{$2\times (N+1)$} matrix which interpolates polynomials at Gauss nodes to values at endpoints.  These two matrices are defined entrywise as
\[
\bm{D}_{ij} = \LRl{\pd{\ell_j}{x}}_{x = x_i}, \qquad \LRp{\bm{V}_f}_{1i} = \ell_i(-1), \qquad  \LRp{\bm{V}_f}_{2i} = \ell_i(1).
\]
We also introduce the diagonal matrix of quadrature weights $\bm{W}_{ij} = \delta_{ij} w_i$, as well as the one-dimensional mass matrix $\bm{M}$ whose entries are $L^2$ inner products of basis functions.  We assume that these inner products are computed using the quadrature rule $(x_i, w_i)$ at which the solution is collocated.  Under such an assumption, the mass matrix is diagonal with entries equal to the quadrature weights
\[
\bm{M}_{ij} = \int_{-1}^1 \ell_i(x)\ell_j(x) \approx \sum_{k=1}^{N+1} \ell_i(x_k)\ell_j(x_k) w_k = \delta_{ij} w_i = \bm{W}_{ij}.
\]
Since $\bm{M} = \bm{W}$ under a collocation assumption, we utilize $\bm{W}$ for the remainder of this work to emphasize that the mass matrix is diagonal and related to the quadrature weights $w_i$.  The treatment of non-diagonal mass matrices is covered in \cite{chan2017discretely, chan2018discretely}.

It can be shown that the mass and differentiation matrices for Gauss nodes fall under the class of generalized SBP (GSBP) operators \cite{fernandez2014generalized}.  
\begin{lemma}
\label{lemma:sbp}
$\bm{Q} = \bm{W}\bm{D}$ satisfies the generalized summation by parts property
\[
\bm{Q} = \bm{V}_f^T \bm{B} \bm{V}_f - \bm{Q}^T, \qquad \bm{B} = \begin{bmatrix}-1 & \\ & 1\end{bmatrix}.
\]
\end{lemma}
The proof is a direct restatement of integration by parts, and can be found in \cite{fernandez2014generalized, ortleb2016kinetic, ortleb2017kinetic, ranocha2018generalised}.  
Lemma~\ref{lemma:sbp} holds for both GLL and Gauss nodes, and switching between these two nodal sets simply results in a redefinition of the matrices $\bm{D}, \bm{V}_f$.  For example, because GLL nodes include boundary points, the interpolation matrix $\bm{V}_f$ reduces to a generalized permutation matrix which extracts the nodal values associated with the left and right endpoints.  

%

\subsection{Existing entropy stable SATs for generalized SBP operators}
\label{sec:gsbpsat}

In this section, we will review the construction of semi-discretely entropy stable discretizations.  Entropy stable discretizations can be constructed by first introducing an entropy conservative scheme, then adding appropriate interface dissipation to produce an entropy inequality.  The construction of entropy conservative schemes relies on the existence of an two-point (dyadic) entropy conservative flux \cite{tadmor1987numerical}.  
\begin{definition}
\label{def:tadmor}
Let $\bm{f}_S(\bm{u}_L,\bm{u}_R)$ be a bivariate function which is symmetric and consistent with the flux function $\bm{f}(\bm{u})$
\begin{align*}
\bm{f}_S(\bm{u}_L,\bm{u}_R) = \bm{f}_S(\bm{u}_R,\bm{u}_L), \qquad \bm{f}_S(\bm{u},\bm{u}) = \bm{f}(\bm{u})
\end{align*}
The numerical flux $\bm{f}_S(\bm{u}_L, \bm{u}_R)$ is entropy conservative if, for entropy variables $\bm{v}_L = \bm{v}(\bm{u}_L), \bm{v}_R = \bm{v}(\bm{u}_R)$, the Tadmor ``shuffle'' condition is satisfied
\begin{align*}
\LRp{\bm{v}_L - \bm{v}_R}^T \bm{f}_S(\bm{u}_L,\bm{u}_R) = (\psi_L - \psi_R), \qquad \psi_L = \psi(\bm{v}(\bm{u}_L)), \quad \psi_R = \psi(\bm{v}(\bm{u}_R)).  
\end{align*}
\end{definition}

For illustrative purposes, we will prove a semi-discrete entropy inequality on a one-dimensional mesh consisting of two elements of degree $N$.  We assume both meshes are translations of a reference element  $[-1,1]$, such that derivatives with respect to physical coordinates are identical to derivatives with respect to reference coordinates.  The extension to multiple elements and variable mesh sizes is straightforward.  

The construction of entropy conservative schemes relies on appropriate SATs for Gauss collocation schemes \cite{fernandez2014review, crean2017high, fernandez2018simultaneous}.  Let the rows of $\bm{V}_f$ be denoted by column vectors $\bm{t}_L, \bm{t}_R$ 
\[
\bm{V}_f = \begin{bmatrix}
\LRp{\bm{t}_L}_1, & \ldots, & \LRp{\bm{t}_L}_{N+1}\\
\LRp{\bm{t}_R}_1, & \ldots, & \LRp{\bm{t}_R}_{N+1}
\end{bmatrix}, \qquad \LRp{\bm{t}_L}_j = \ell_j(-1), \qquad \LRp{\bm{t}_R}_j = \ell_j(1).
\]
The inter-element coupling terms in \cite{fernandez2014review, crean2017high, fernandez2018simultaneous} utilize a decomposition of the surface matrix $\bm{V}_f^T\bm{B}\bm{V}_f$ as  
\begin{equation}
\bm{V}_f^T\bm{B}\bm{V}_f 
= \bm{t}_R\bm{t}_R^T - \bm{t}_L\bm{t}_L^T.
\label{eq:bmatdecomp}
\end{equation}
The construction of entropy conservative schemes on multiple elements requires appropriate inter-element coupling terms (SATs) involving $\bm{t}_L, \bm{t}_R$.  We consider a two element mesh, and show how when coupled with SATs, the resulting discretization matrices can be interpreted as constructing a global SBP operator.  

Let $\bm{u}^1_N, \bm{u}^2_N$ denote nodal degrees of freedom of the vector valued solution $\bm{u}(x)$ on the first and second element, respectively.  To simplify notation, we assume that all following operators are defined in terms of Kronecker products \cite{chen2017entropy}, such that they are applied to each component of $\bm{u}^1_N, \bm{u}^2_N$.  
We first define the matrix 
\begin{equation}
\bm{S} \myeq \bm{Q} - \frac{1}{2}\bm{V}_f^T\bm{B}\bm{V}_f.
\label{eq:Smat}
\end{equation}
It is straightforward to show (using Lemma~\ref{lemma:sbp}) that $\bm{S}$ is skew-symmetric.  We can now define an SBP operator $\bm{D}_h = \bm{W}_h^{-1}\bm{Q}_h$ over two elements
\begin{equation}
\bm{Q}_{h} \myeq 
\underbrace{\begin{bmatrix}
\bm{S}  & \frac{1}{2}\bm{t}_R\bm{t}_L^T \\
- \frac{1}{2}\bm{t}_L\bm{t}_R^T & \bm{S}
\end{bmatrix}}_{\bm{S}_h}
+
\underbrace{\begin{bmatrix}
-\frac{1}{2}\bm{t}_L\bm{t}_L^T  & \\
 & \frac{1}{2} \bm{t}_R\bm{t}_R^T
\end{bmatrix}}_{\frac{1}{2}\bm{B}_h}
, \qquad 
\bm{W}_h \myeq \begin{bmatrix}\bm{W} &\\
& \bm{W}
\end{bmatrix}.
\label{eq:gsbp}
\end{equation}
It can be shown that $\bm{D}_h$ is high order accurate such that, if $\bm{u}_h$ is a polynomial of degree $N$, it is differentiated exactly.  Straightforward computations show that $\bm{Q}_h$ also satisfies an SBP property $\bm{Q}_h + \bm{Q}_h^T = \bm{B}_h$.

Ignoring boundary conditions, an entropy conservative scheme for (\ref{eq:nonlinpde}) on two elements can then be given as
\begin{gather}
\td{}{t}\bm{W}_h\bm{u}_h + 2\LRp{\bm{Q}_h \circ \bm{F}_S}\bm{1} = 0, \qquad \bm{u}_h = \begin{bmatrix}
\bm{u}^1_N\\
\bm{u}^2_N
\end{bmatrix} \label{eq:ESgsbp} \\
\LRp{\bm{F}_S}_{ij} = \bm{f}_S\LRp{\LRp{\bm{u}_h}_i,\LRp{\bm{u}_h}_j}, \qquad 1 \leq  i,j \leq 2(N+1), \nonumber
\end{gather}
where $\circ$ denotes the Hadamard product \cite{horn2012matrix}.  {It should be emphasized that here, $(\bm{u}_h)_i, (\bm{u}_h)_j$ denote vectors containing solution components at nodes $i,j$, and that (because $\bm{f}_S$ is a vector-valued flux) the term $\LRp{\bm{F}_S}_{ij}$ should be interpreted as a diagonal matrix whose diagonal entries consist of the components of $\bm{f}_S\LRp{(\bm{u}_h)_i, (\bm{u}_h)_j}$.}

Multiplying (\ref{eq:ESgsbp}) by $\bm{v}_h^T = \bm{v}\LRp{\bm{u}_h}^T$ 
will yield a semi-discrete version of the conservation of entropy (mimicking (\ref{eq:weakentropyineq}) with the inequality replaced by an equality)
\begin{equation}
\label{eq:consentropyGSBP}
\td{}{t}\bm{W}_h S(\bm{u}_h) + \bm{v}_h^T\LRp{\bm{B}_h\circ \bm{F}_S}\bm{1} -\bm{1}^T\bm{B}_h\psi\LRp{\bm{u}_h} = 0.
\end{equation}
We refer to \cite{crean2017high, crean2018entropy} for the proof of (\ref{eq:consentropyGSBP}).  

The drawback of the SATs introduced in this section lies in the nature of the off-diagonal matrices $\bm{t}_R\bm{t}_L$ and $-\bm{t}_L\bm{t}_R$.  For Gauss nodes, these blocks are dense, which implies that inter-element coupling terms produce a non-compact stencil.  Evaluating (\ref{eq:ESgsbp}) requires computing two-point fluxes $\bm{f}_S$ between all nodes on two neighboring elements, which significantly increases both the computational work, as well as communication between neighboring elements.  This leads to all-to-all coupling between degrees of freedom in 1D, and to coupling along one-dimensional lines of nodes in higher dimensions due to the tensor product structure.  

The main goal of this work is to circumvent this tighter coupling of degrees of freedom introduced by the SATs described in this section, which can be done through the use of ``decoupled'' SBP operators.  

\subsection{Decoupled SBP operators}

Decoupled SBP operators were first introduced in \cite{chan2017discretely} and used to construct entropy stable schemes on simplicial elements.  These operators (and simplifications under a collocation assumption) are presented in a more general setting in \cite{chan2017discretely, chan2018discretely} and in Appendix~\ref{app:decoupled}.  In this section, decoupled SBP operators are introduced in one dimension for GLL and Gauss nodal sets.  

Decoupled SBP operators build upon the GSBP matrices $\bm{W}, \bm{Q}$, interpolation matrix $\bm{V}_f$, and boundary matrix $\bm{B}$ introduced in Section~\ref{sec:gsbp}.  The decoupled SBP operator $\bm{Q}_N$ is defined as the block matrix 
\begin{equation}
\bm{Q}_N \myeq \begin{bmatrix}
\bm{Q} - \frac{1}{2}\bm{V}_f^T\bm{B}\bm{V}_f & \frac{1}{2}\bm{V}_f^T\bm{B}\\
-\frac{1}{2}\bm{B}\bm{V}_f & \frac{1}{2}\bm{B}
\end{bmatrix}.
\label{eq:qndef}
\end{equation}
Lemma~\ref{lemma:sbp} and straightforward computations show that $\bm{Q}_N$ also satisfies the following SBP property
\begin{lemma}
\label{lemma:dsbp}
Let $\bm{Q}_N$ be defined through (\ref{eq:qndef}).  Then,
\[
\bm{Q}_N + \bm{Q}_N^T = \begin{bmatrix}
\bm{0} &\\
& \bm{B}
\end{bmatrix}.
\]
\end{lemma}

We note that the matrix $\bm{Q}_N$ acts not only on volume nodes, but on both volume and surface nodes.  Thus, it is not immediately clear how to apply this operator to GSBP discretizations of nonlinear conservation laws.  It is straightforward to evaluate the nonlinear flux at volume nodes since the solution is collocated at these points; however, evaluating the nonlinear flux at surface nodes is less straightforward.  Moreover, $\bm{Q}_N$ does not directly define a difference operator, and must be combined with another operation to produce an approximation to the derivative.  We will discuss how to apply $\bm{Q}_N$ in two steps.  First, we will show how to approximate the derivative of an arbitrary function using $\bm{Q}_N$ given function values at both volume and surface nodes.  Then, we will describe how to apply this approximation to compute derivatives of nonlinear flux functions given collocated solution values at volume nodes.    

Let $f(x), g(x)$ denote two functions, and let $\bm{f}, \bm{g}$ denote the values of $f,g$ at interior nodal points.  We also define vectors $\bm{f}_N, \bm{g}_N$ denoting the values of $f,g$ at both interior and boundary points 
\begin{equation}
\bm{f}_N = \begin{bmatrix}
f\LRp{x_1}\\
\vdots\\
f\LRp{x_{N+1}}\\
f(-1)\\
f(1)
\end{bmatrix} = \begin{bmatrix}
\bm{f} \\
\bm{f}_f
\end{bmatrix}, \qquad
\bm{g}_N = \begin{bmatrix}
g\LRp{x_1}\\
\vdots\\
g\LRp{x_{N+1}}\\
g(-1)\\
g(1)
\end{bmatrix} = \begin{bmatrix}
\bm{g} \\
\bm{g}_f
\end{bmatrix}.
\label{eq:fg}
\end{equation}
Then, a polynomial approximation to $f\pd{g}{x}$ can be computed using $\bm{Q}_N$.  Let $\bm{u}$ denote the nodal values of the polynomial $u(x) \approx f\pd{g}{x}$.  These coefficients are computed via
\begin{equation}
\bm{W}\bm{u} = \begin{bmatrix}
\bm{I}\\
\bm{V}_f
\end{bmatrix}^T \diag{\bm{f}_N} \bm{Q}_N \bm{g}_N.
\label{eq:qn}
\end{equation}
The approximation (\ref{eq:qn}) can be rewritten in ``strong'' form as follows
\begin{align*}
\bm{u} = \bm{W}^{-1}\begin{bmatrix}
\bm{I}\\
\bm{V}_f
\end{bmatrix}^T& \diag{\bm{f}_N} \bm{Q}_N \bm{g}_N \\
= \diag{\bm{f}}\bm{D}\bm{g} &+ \frac{1}{2}\diag{\bm{f}} \bm{W}^{-1} \bm{V}_f^T \bm{B}\LRp{\bm{g}_f - \bm{V}_f\bm{g}} \\
&+ \frac{1}{2}\bm{W}^{-1} \bm{V}_f^T\bm{B}\diag{\bm{f}_f} \LRp{\bm{g}_f - \bm{V}_f\bm{g}},
\end{align*}
where we have used the fact that diagonal matrices commute to simplify expressions.  The decoupled SBP operator $\bm{Q}_N$ can thus be interpreted as adding boundary corrections to the GSBP operator $\bm{D}$ in a skew-symmetric fashion.  \note{More specifically, the expression (\ref{eq:qn}) corresponds to a quadrature approximation of the following variational approximation of the derivative \cite{chan2017discretely}: find $u\in P^N([-1,1])$ such that},
\[
\note{\int_{-1}^1u({x})v({x})  = \int_{-1}^1 {I_N f\pd{I_Ng}{x}v} + \LRl{(g-I_Ng)\frac{\LRp{fv + I_N(fv)}}{2}}_{-1}^1,} \qquad \forall v\in P^N([-1,1]),
\]
\note{where $I_N$ denotes the degree $N$ polynomial approximation at the $(N+1)$ Gauss points.  }

The approximation (\ref{eq:qn}) can also be applied to Gauss collocation schemes for nonlinear conservation laws. Let $u \in P^N$ be represented by the vector $\bm{u}$ of values at Gauss points, and let $f(x), g(x)$ denote two nonlinear functions.   The operator $\bm{Q}_N$ can be used to approximate the quantity $f(u)\pd{g(u)}{x}$ using (\ref{eq:qn}) if $\bm{f}, \bm{g}$ are defined as
\[
\bm{f}_N = \begin{bmatrix}
f\LRp{\bm{u}_1}\\
\vdots\\
f\LRp{\bm{u}_{N+1}}\\
f\LRp{\bm{t}_L^T\bm{u}}\\
f\LRp{\bm{t}_R^T\bm{u}}
\end{bmatrix} = \begin{bmatrix}
f(\bm{u}) \\
f(\bm{u}_f)
\end{bmatrix}, \qquad
\bm{g}_N = \begin{bmatrix}
g\LRp{\bm{u}_1}\\
\vdots\\
g\LRp{\bm{u}_{N+1}}\\
g\LRp{\bm{t}_L^T\bm{u}}\\
g\LRp{\bm{t}_R^T\bm{u}}
\end{bmatrix} = \begin{bmatrix}
g(\bm{u}) \\
g(\bm{u}_f)
\end{bmatrix}, \qquad \bm{u}_f = \bm{V}_f\bm{u}.
\]

It was shown in \cite{chan2017discretely} that $\bm{u}$ is a high order accurate approximation to the quantity $f\pd{g}{x}$.  Both the generalized SBP operator $\bm{D}$ and the expression in (\ref{eq:qn}) involving the decoupled SBP operator recover exact derivatives of high order polynomials.  However, when applied to non-polynomial functions, the decoupled SBP operator $\bm{Q}_N$ improves accuracy near the boundaries.  Figure~\ref{fig:dsbpcorrect} illustrates this by using both operators to approximate the derivative of a Gaussian $e^{-4x^2}$ on $[-1,1]$.  The decoupled SBP operator results in an improved approximation at all orders of approximation.  

\begin{figure}
\centering
\subfloat[$N=5$ approximation]{
\label{subfig:dsbp1}
\begin{tikzpicture}
\begin{axis}[
    width=.51\textwidth,
    xlabel={$x$ coordinate},
    ymin=-2, ymax=2.5,
    legend pos=north east, legend cell align=left, legend style={font=\tiny},	
    xmajorgrids=true, ymajorgrids=true, grid style=dashed,
    legend entries={GSBP, Decoupled, Exact}    
]
\pgfplotsset{
cycle list={{blue, only marks, mark=*}, {blue}, {red, only marks,mark=square*},{red},{black,dashed}}
}
\addlegendimage{blue, mark=*}
\addlegendimage{red, mark=square*}
\addlegendimage{black, dashed}

\addplot+[semithick, mark options={solid, fill=markercolor}]
coordinates{(-0.93247,-0.677965)(-0.661209,1.35632)(-0.238619,1.07263)(0.238619,-1.07263)(0.661209,-1.35632)(0.93247,0.677965)};

\addplot+[semithick, mark options={solid, fill=markercolor}]
coordinates{(-1,-1.56577)(-0.959184,-1.00891)(-0.918367,-0.51366)(-0.877551,-0.0774098)(-0.836735,0.302464)(-0.795918,0.628585)(-0.755102,0.903575)(-0.714286,1.13006)(-0.673469,1.31065)(-0.632653,1.44798)(-0.591837,1.54466)(-0.55102,1.60333)(-0.510204,1.62659)(-0.469388,1.61708)(-0.428571,1.57742)(-0.387755,1.51022)(-0.346939,1.41812)(-0.306122,1.30372)(-0.265306,1.16966)(-0.22449,1.01856)(-0.183673,0.85303)(-0.142857,0.675704)(-0.102041,0.489201)(-0.0612245,0.296143)(-0.0204082,0.0991513)(0.0204082,-0.0991513)(0.0612245,-0.296143)(0.102041,-0.489201)(0.142857,-0.675704)(0.183673,-0.85303)(0.22449,-1.01856)(0.265306,-1.16966)(0.306122,-1.30372)(0.346939,-1.41812)(0.387755,-1.51022)(0.428571,-1.57742)(0.469388,-1.61708)(0.510204,-1.62659)(0.55102,-1.60333)(0.591837,-1.54466)(0.632653,-1.44798)(0.673469,-1.31065)(0.714286,-1.13006)(0.755102,-0.903575)(0.795918,-0.628585)(0.836735,-0.302464)(0.877551,0.0774098)(0.918367,0.51366)(0.959184,1.00891)(1,1.56577)};

\addplot+[semithick, mark options={solid, fill=markercolor}]
coordinates{(-0.93247,-0.194074)(-0.661209,1.16978)(-0.238619,1.20913)(0.238619,-1.20913)(0.661209,-1.16978)(0.93247,0.194074)};

\addplot+[semithick, mark options={solid, fill=markercolor}]
coordinates{(-1,-0.434491)(-0.959184,-0.303266)(-0.918367,-0.130624)(-0.877551,0.0698403)(-0.836735,0.286078)(-0.795918,0.507515)(-0.755102,0.724981)(-0.714286,0.930644)(-0.673469,1.11794)(-0.632653,1.28151)(-0.591837,1.41712)(-0.55102,1.5216)(-0.510204,1.59278)(-0.469388,1.62942)(-0.428571,1.63111)(-0.387755,1.59825)(-0.346939,1.53198)(-0.306122,1.43404)(-0.265306,1.30681)(-0.22449,1.15314)(-0.183673,0.976347)(-0.142857,0.78012)(-0.102041,0.568459)(-0.0612245,0.345602)(-0.0204082,0.115959)(0.0204082,-0.115959)(0.0612245,-0.345602)(0.102041,-0.568459)(0.142857,-0.78012)(0.183673,-0.976347)(0.22449,-1.15314)(0.265306,-1.30681)(0.306122,-1.43404)(0.346939,-1.53198)(0.387755,-1.59825)(0.428571,-1.63111)(0.469388,-1.62942)(0.510204,-1.59278)(0.55102,-1.5216)(0.591837,-1.41712)(0.632653,-1.28151)(0.673469,-1.11794)(0.714286,-0.930644)(0.755102,-0.724981)(0.795918,-0.507515)(0.836735,-0.286078)(0.877551,-0.0698403)(0.918367,0.130624)(0.959184,0.303266)(1,0.434491)};

\addplot+[semithick, mark options={solid, fill=markercolor}]
coordinates{(-1,0.146525)(-0.959184,0.193522)(-0.918367,0.251752)(-0.877551,0.322529)(-0.836735,0.406852)(-0.795918,0.50522)(-0.755102,0.617438)(-0.714286,0.742415)(-0.673469,0.877994)(-0.632653,1.02082)(-0.591837,1.1663)(-0.55102,1.30861)(-0.510204,1.44089)(-0.469388,1.55553)(-0.428571,1.64452)(-0.387755,1.70003)(-0.346939,1.71492)(-0.306122,1.68342)(-0.265306,1.60163)(-0.22449,1.46805)(-0.183673,1.2839)(-0.142857,1.05327)(-0.102041,0.783025)(-0.0612245,0.482507)(-0.0204082,0.162994)(0.0204082,-0.162994)(0.0612245,-0.482507)(0.102041,-0.783025)(0.142857,-1.05327)(0.183673,-1.2839)(0.22449,-1.46805)(0.265306,-1.60163)(0.306122,-1.68342)(0.346939,-1.71492)(0.387755,-1.70003)(0.428571,-1.64452)(0.469388,-1.55553)(0.510204,-1.44089)(0.55102,-1.30861)(0.591837,-1.1663)(0.632653,-1.02082)(0.673469,-0.877994)(0.714286,-0.742415)(0.755102,-0.617438)(0.795918,-0.50522)(0.836735,-0.406852)(0.877551,-0.322529)(0.918367,-0.251752)(0.959184,-0.193522)(1,-0.146525)};
\end{axis}
\end{tikzpicture}
}
\subfloat[$L^2$ errors]{
\label{subfig:dsbp2}
\begin{tikzpicture}
\begin{semilogyaxis}[
    width=.51\textwidth,
    xlabel={Degree $N$},
    ylabel={$L^2$ errors}, 
    legend pos=north east, legend cell align=left, legend style={font=\tiny},	
    xmajorgrids=true, ymajorgrids=true, grid style=dashed,
    legend entries={GSBP, Decoupled SBP}    
]
\pgfplotsset{
cycle list={{blue, mark=*}, {red, mark=square*}}
}

\addplot+[semithick, mark options={solid, fill=markercolor}]
coordinates{(1,1.67087)(2,1.84788)(3,1.29689)(4,1.14406)(5,0.67189)(6,0.4428)(7,0.242532)(8,0.124111)(9,0.0657786)(10,0.0272413)(11,0.0141812)(12,0.00491087)(13,0.00252903)(14,0.000750851)(15,0.000383985)};

\addplot+[semithick, mark options={solid, fill=markercolor}]
coordinates{(1,1.37796)(2,1.23235)(3,0.769706)(4,0.682916)(5,0.348496)(6,0.252957)(7,0.120617)(8,0.0691975)(9,0.0323286)(10,0.0149469)(11,0.00695358)(12,0.00266351)(13,0.00124093)(14,0.000403653)(15,0.000188708)};

\end{semilogyaxis}
\end{tikzpicture}

}

\caption{Approximations of derivatives of a Gaussian $e^{-4x^2}$ using the generalized SBP operator $\bm{D}$ and the decoupled SBP operator $\bm{Q}_N$ via (\ref{eq:qn}).  In Figure~\ref{subfig:dsbp1}, the colored circles and squares denote values at Gauss points for a degree $N = 5$ approximation.  Figure~\ref{subfig:dsbp2} shows the convergence of $L^2$ errors as $N$ increases. }
\label{fig:dsbpcorrect}
\end{figure}
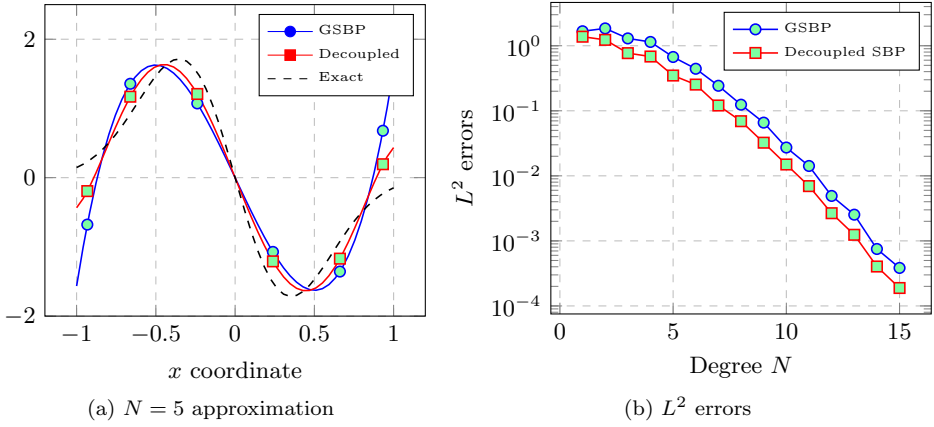

\subsection{An entropy stable Gauss collocation scheme based on decoupled SBP operators}

We can now construct an entropy conservative Gauss collocation scheme with compact SATs using decoupled SBP operators.  As in Section~\ref{sec:gsbpsat}, we will construct a Gauss collocation scheme and provide a proof of semi-discrete entropy conservation for a two-element mesh.  

We first note that $\bm{B}$ can be trivially decomposed into the sum of two outer products as in (\ref{eq:bmatdecomp})
\begin{equation}
\bm{B} = \begin{bmatrix}
-1 & 0 \\
0 & 1
\end{bmatrix} = \bm{e}_R\bm{e}_R^T - \bm{e}_L\bm{e}_L^T, \qquad \bm{e}_L = \begin{bmatrix} 1\\0\end{bmatrix}, \qquad \bm{e}_R = \begin{bmatrix} 0\\1\end{bmatrix}.
\label{eq:bmatdecomp2}
\end{equation}
The vectors $\bm{e}_L, \bm{e}_R$ are related to $\bm{t}_L, \bm{t}_R$ through the interpolation matrix $\bm{V}_f$.  Because $\bm{t}_L, \bm{t}_R$ are rows of $\bm{V}_f$, $\bm{t}_L\bm{1} = \bm{t}_R\bm{1} = 1$.  This can be used to show, for example, that
\begin{equation}
\label{eq:et}
\bm{B}\bm{V}_f\bm{1} = \LRp{\bm{e}_R\bm{e}_R^T - \bm{e}_L\bm{e}_L^T}\bm{1} = \bm{e}_R-\bm{e}_L.
\end{equation}

We can define a decoupled SBP matrix $\bm{Q}_h$ over two elements as follows
\begin{align}
\bm{Q}_h \myeq& \LRs{\begin{array}{cc|cc}
\bm{S} &  \frac{1}{2}\bm{V}_f^T\bm{B} & & \\
 -\frac{1}{2}\bm{B}\bm{V}_f & -\frac{1}{2}\bm{e}_L\bm{e}_L^T & & \frac{1}{2}\bm{e}_R\bm{e}_L^T\\ [.2em]
 \hline\\[-1.0em]
&& \bm{S} & \frac{1}{2}\bm{V}_f^T\bm{B} \\
& -\frac{1}{2}\bm{e}_L\bm{e}_R^T & - \frac{1}{2}\bm{B}\bm{V}_f & \frac{1}{2}\bm{e}_R\bm{e}_R^T
\end{array}}, \label{eq:decoupledsbp2elem}
\end{align}
where we have abused notation and redefined $\bm{Q}_h$.  We can also show that $\bm{Q}_h\bm{1} = 0$.  Using (\ref{eq:et}), we have that
\begin{align}
\bm{Q}_h\bm{1} &= \begin{bmatrix}
\LRp{\bm{S} + \frac{1}{2}\bm{V}_f^T\bm{B} }\bm{1}\\
-\frac{1}{2}\bm{B}\bm{V}_f\bm{1} + \frac{1}{2}\LRp{\bm{e}_R-\bm{e}_L}\\
\LRp{\bm{S} + \frac{1}{2}\bm{V}_f^T\bm{B} }\bm{1}\\
-\frac{1}{2}\bm{B}\bm{V}_f\bm{1} + \frac{1}{2}\LRp{\bm{e}_R-\bm{e}_L}
\end{bmatrix} 
= 
\begin{bmatrix}
\LRp{\bm{Q} - \frac{1}{2}\bm{V}_f^T\bm{B}\bm{V}_f + \frac{1}{2}\bm{V}_f^T\bm{B}\bm{V}_f}\bm{1}\\
\frac{1}{2}\bm{B}\LRp{-\bm{1} + \bm{1}}\\
\LRp{\bm{Q} - \frac{1}{2}\bm{V}_f^T\bm{B}\bm{V}_f + \frac{1}{2}\bm{V}_f^T\bm{B}\bm{V}_f}\bm{1}\\
\frac{1}{2}\bm{B}\LRp{-\bm{1} + \bm{1}}\\
\end{bmatrix} = \bm{0}.
\label{eq:Q1zero}
\end{align}
Here, we have used the definition of $\bm{B}$ in (\ref{eq:bmatdecomp2}), the fact that $\bm{V}_f\bm{1} = \bm{1}$ and $\bm{Q}\bm{1} = 0$ \cite{fernandez2014generalized}, and the definition of $\bm{S}$ in (\ref{eq:Smat}).  This property (\ref{eq:Q1zero}) will be used in the proof of entropy conservation.  

It can be helpful to split up $\bm{Q}_h$ into two matrices
\begin{gather*}
\bm{Q}_h = \bm{S}_h + \frac{1}{2}\bm{B}_h, \qquad
\bm{S}_h \myeq \LRs{\begin{array}{cc|cc}
\bm{S} &  \frac{1}{2}\bm{V}_f^T\bm{B} & & \\
 - \frac{1}{2}\bm{B}\bm{V}_f &  & & \frac{1}{2}\bm{e}_R\bm{e}_L^T\\ [.2em]
 \hline\\[-1.0em]
&& \bm{S} &  \frac{1}{2}\bm{V}_f^T\bm{B} \\
& -\frac{1}{2}\bm{e}_L\bm{e}_R^T &  - \frac{1}{2}\bm{B}\bm{V}_f &\\
\end{array}}\\
\bm{B}_h \myeq \LRs{\begin{array}{cc|cc}
&  & & \\
 & -\bm{e}_L\bm{e}_L^T & & \\
 \hline
&&  &  \\
&  && \bm{e}_R\bm{e}_R^T\\
\end{array}}.
\end{gather*}
The matrix $\bm{S}_h$ is skew-symmetric, while the matrix $\bm{B}_h$ functions as a boundary operator which extracts boundary values over the two-element domain (i.e.\ $\bm{B}_h$ extracts a left boundary value from element 1 and a right boundary value from element 2).  Note that here, $\bm{B}_h$ is diagonal, unlike the boundary operator defined in Section~\ref{sec:gsbpsat}.  It should also be noted that the two-element decoupled SBP operator $\bm{Q}_h$ in (\ref{eq:decoupledsbp2elem}) is related to the GSBP operator in (\ref{eq:gsbp}) through a block interpolation matrix
\[
\bm{I}_h \myeq \begin{bmatrix}
\bm{I} &\\
\bm{V}_f &\\
&\bm{I}\\
&\bm{V}_f\\
\end{bmatrix}, \qquad \bm{I}_h^T\bm{Q}_h\bm{I}_h = {\begin{bmatrix}
\bm{S}  & \frac{1}{2}\bm{t}_R\bm{t}_L^T \\
- \frac{1}{2}\bm{t}_L\bm{t}_R^T & \bm{S}
\end{bmatrix}}
+
{\begin{bmatrix}
-\frac{1}{2}\bm{t}_L\bm{t}_L^T  & \\
 & \frac{1}{2} \bm{t}_R\bm{t}_R^T
\end{bmatrix}}.
\]

We will now utilize the two-element operators described in this section to construct an entropy stable Gauss collocation scheme on two elements.  This scheme will differ from that of Section~\ref{sec:gsbpsat} in that neighboring elements will only be coupled together through face values.  
Let $\bm{u}^1_N, \bm{u}^2_N$ denote the values of the conservative variables at Gauss points on elements 1 and 2, respectively, and let $\bm{u}_h$ denote their concatenation as defined in (\ref{eq:ESgsbp}).  Let $\bm{v}(\bm{u}^i_N)$ denote the evaluation of entropy variables at Gauss points on element $i$, and define the ``entropy-projected conservative variables'' $\tilde{\bm{u}}^1_f, \tilde{\bm{u}}^2_f$ by evaluating the conservative variables in terms of the  interpolated values of the entropy variables at element boundaries 
\[
\tilde{\bm{u}}^i_f \myeq \bm{u}(\bm{v}^i_f), \qquad \bm{v}^i_f \myeq \bm{V}_f\bm{v}(\bm{u}^i_N), \qquad i = 1,2. 
\]
We now introduce $\bm{u}_h = \begin{bmatrix}\bm{u}^1_N & \bm{u}^2_N\end{bmatrix}^T$ as the concatenated vector of solution values at Gauss points, and define $\bm{v}_h, \tilde{\bm{v}}$, and $\tilde{\bm{u}}$ as follows:  
\[
\bm{v}_h \myeq \bm{v}\LRp{\bm{u}_h} = \bm{v}\LRp{\begin{bmatrix}
{\bm{u}^1_N}\\
{\bm{u}^2_N}
\end{bmatrix}}, \qquad \tilde{\bm{v}} \myeq \bm{I}_h\bm{v}_h, \qquad \tilde{\bm{u}} \myeq \bm{u}\LRp{\tilde{\bm{v}}} = 
\begin{bmatrix}
\bm{u}^1_N \\
\tilde{\bm{u}}^1_f \\
\bm{u}^2_N \\
\tilde{\bm{u}}^2_f
\end{bmatrix}.  
\]
The term $\bm{v}_h$ corresponds to the vector of entropy variables at $\bm{u}_h$, while $\tilde{\bm{v}}$ corresponds to the concatenated vector of Gauss point values and interpolated boundary values $\bm{v}^i_f$ of the entropy variables.  The term $\tilde{\bm{u}}$ denotes the evaluation of the conservative variables in terms of $\tilde{\bm{v}}$.  

Finally, we define $\bm{F}_S$ as the matrix of evaluations of the two-point flux $\bm{f}_S$ at combinations of values of $\tilde{\bm{u}}$
\[
\LRp{\bm{F}_S}_{ij} \myeq \bm{f}_S\LRp{\tilde{\bm{u}}_i,\tilde{\bm{u}}_j}, \quad 1\leq i,j \leq 2(N+3).
\]
Note that, due to the consistency of $\bm{f}_S$, the diagonal of $\bm{F}_S$ reduces to flux evaluations
\begin{equation}
\label{eq:diagFS}
\LRp{\bm{F}_S}_{ii} = \bm{f}_S\LRp{\tilde{\bm{u}}_i,\tilde{\bm{u}}_i} = \bm{f}\LRp{\tilde{\bm{u}}_i}.  
\end{equation}
We can now construct a semi-discretely entropy conservative formulation based on decoupled SBP operators:
\begin{theorem}
\label{thm:consentropy}
Let $\bm{Q}_h$ be defined by (\ref{eq:decoupledsbp2elem}), and let $\bm{u}_h$ denote the two-element solution of the following formulation:
\begin{gather}
\label{eq:form}
\bm{W}_h \td{}{t} \bm{u}_h + \note{2} \bm{I}_h^T\LRp{\bm{Q}_h \circ \bm{F}_S}\bm{1} = 0.
\end{gather}
Then, $\bm{u}_h$ satisfies a semi-discrete conservation of entropy
\begin{gather*}
\bm{1}^T\bm{W}_h\td{S(\bm{u}_h)}{t} + \bm{1}^T\bm{B}\LRp{\bm{v}_f^T\bm{f}(\tilde{\bm{u}}_f) - \psi\LRp{\tilde{\bm{u}}_f}} = 0\\
\bm{v}_f \myeq \begin{bmatrix}
\bm{t}_L^T \bm{v}\LRp{\bm{u}_N^1}\\
\bm{t}_R^T \bm{v}\LRp{\bm{u}_N^2}
\end{bmatrix}, \qquad \tilde{\bm{u}}_f \myeq \bm{u}(\bm{v}_f).
\end{gather*}
\end{theorem}
\begin{proof}
The proof results from testing with $\bm{v}_h^T$.  Since $\bm{v}(\bm{u}) = \pd{S(\bm{u})}{\bm{u}}$ and $\bm{W}_h$ is diagonal, the time term yields
\[
\bm{v}_h^T\td{}{t}\bm{W}_h\bm{u}_h = \bm{1}^T \td{}{t}\bm{W}_h\diag{\bm{v}_h}\bm{u}_h = \bm{1}^T\td{}{t}\bm{W}_h S(\bm{u}_h).
\]
The spatial term can be manipulated as follows
\begin{align*}
2\bm{v}_h^T \bm{I}_h^T\LRp{\bm{Q}_h \circ \bm{F}_S}\bm{1} &= 2\LRp{\bm{I}_h\bm{v}_h}^T \LRp{\LRp{\bm{S}_h \circ \bm{F}_S}\bm{1} + \frac{1}{2}\LRp{\bm{B}_h \circ \bm{F}_S}\bm{1}}\\
&= \tilde{\bm{v}}^T\LRp{\bm{B}_h \circ \bm{F}_S}\bm{1} + \tilde{\bm{v}}^T \LRp{\bm{S}_h \circ \bm{F}_S}\bm{1} - \bm{1}^T \LRp{\bm{S}_h \circ \bm{F}_S}\tilde{\bm{v}},
\end{align*}
where we have used the skew-symmetry of $\bm{S}_h$ in the last step.  The boundary term reduces to 
\[
\tilde{\bm{v}}^T\LRp{\bm{B}_h \circ \bm{F}_S}\bm{1} = \bm{1}^T\bm{B}_h\tilde{\bm{v}}^T\bm{f}\LRp{\tilde{\bm{u}}}= \bm{1}^T\bm{B}\tilde{\bm{v}}_f^T\bm{f}\LRp{\tilde{\bm{u}_f}},
\]
where we have used (\ref{eq:diagFS}) and the fact that $\bm{B}_h$ is diagonal.  Here, $\tilde{\bm{v}}^T\bm{f}\LRp{\tilde{\bm{u}}}$, $\tilde{\bm{v}}_f^T\bm{f}\LRp{\tilde{\bm{u}}_f}$ denote vectors whose entries are the vector inner products of components of $\tilde{\bm{v}}$ and $\bm{f}\LRp{\tilde{\bm{u}}}$ at volume and face points, respectively.  

The volume terms can be manipulated using the definition of $\bm{F}_S$ and the Tadmor shuffle condition in Definition~\ref{def:tadmor}
\begin{align*}
\tilde{\bm{v}}^T \LRp{\bm{S}_h \circ \bm{F}_S} \bm{1} - \bm{1}^T \LRp{\bm{S}_h \circ \bm{F}_S}\tilde{\bm{v}} 
&= \sum_{ij} \LRp{\bm{S}_h}_{ij} \LRp{\tilde{\bm{v}}_i - \tilde{\bm{v}}_j}^T\bm{f}_S\LRp{\tilde{\bm{u}}_i,\tilde{\bm{u}}_j}\\
&= \sum_{ij} \LRp{\bm{S}_h}_{ij} \LRp{\psi\LRp{\tilde{\bm{u}}_i} - \psi\LRp{\tilde{\bm{u}}_j}}\\
&= \psi\LRp{\tilde{\bm{u}}}^T \bm{S}_h \bm{1}- \bm{1}^T \bm{S}_h \psi\LRp{\tilde{\bm{u}}} = 2\LRp{ \psi\LRp{\tilde{\bm{u}}}^T\bm{S}_h\bm{1}},
\end{align*}
where we have again used the skew-symmetry of $\bm{S}_h$ in the last step.  Substituting $\bm{S}_h = \bm{Q}_h - \frac{1}{2}\bm{B}_h$ and using (\ref{eq:Q1zero}) yields
\begin{align*} 
2\LRp{ \psi\LRp{\tilde{\bm{u}}}^T\bm{S}_h\bm{1}} &= \psi\LRp{\tilde{\bm{u}}}^T\LRp{2\bm{Q}_h - \bm{B}_h} \bm{1} \\
&=-\psi\LRp{\tilde{\bm{u}}}^T\bm{B}_h \bm{1}  = -\bm{1}^T\bm{B}_h \psi\LRp{\tilde{\bm{u}}} =  -\bm{1}^T\bm{B} \psi\LRp{\tilde{\bm{u}}_f}, 
\end{align*} 
where we have used the symmetry of $\bm{B}_h$ in the second to last step.
\end{proof}

\begin{remark}
The proof of Theorem~\ref{thm:consentropy} follows directly from choosing either GLL or Gauss quadratures in Theorem 4 of \cite{chan2017discretely}.  The proof is reproduced here for clarity, as the two-element case illuminates the skew-symmetric nature and structure of the inter-element coupling more explicitly.  
\end{remark}

\note{Extending the two-element case to multiple elements can be done by defining analogous ``global'' differentiation matrices.  Suppose that there are now three elements.  Then, the matrices $\bm{S}_h, \bm{B}_h$ can be defined as}
\begin{gather*}
\note{\bm{S}_h \myeq \LRs{\begin{array}{cc|cc|cc}
\bm{S} &  \frac{1}{2}\bm{V}_f^T\bm{B} & &&& \\
 - \frac{1}{2}\bm{B}\bm{V}_f &  & & \frac{1}{2}\bm{e}_R\bm{e}_L^T &&\\ [.2em]
 \hline\\[-1.0em]
&& \bm{S} &  \frac{1}{2}\bm{V}_f^T\bm{B}&& \\
& -\frac{1}{2}\bm{e}_L\bm{e}_R^T &  - \frac{1}{2}\bm{B}\bm{V}_f &&&\frac{1}{2}\bm{e}_R\bm{e}_L^T\\
 \hline\\[-1.0em]
&&&& \bm{S} &  \frac{1}{2}\bm{V}_f^T\bm{B} \\
&&& -\frac{1}{2}\bm{e}_L\bm{e}_R^T &  - \frac{1}{2}\bm{B}\bm{V}_f &\\
\end{array}}}
\\
\note{\bm{B}_h \myeq \LRs{\begin{array}{cc|c|cc}
&  & && \\
 & -\bm{e}_L\bm{e}_L^T & && \\
 \hline
&&  \bm{0} & & \\
 \hline
&&&&    \\
&&&  & \bm{e}_R\bm{e}_R^T\\
\end{array}}, \qquad \bm{I}_h = \begin{bmatrix}
\bm{I} &&\\
\bm{V}_f &&\\
&\bm{I}&\\
&\bm{V}_f&\\
&&\bm{I}\\
&&\bm{V}_f\\
\end{bmatrix}.}
\end{gather*}
\note{These global operators again satisfy $\bm{Q}_h\bm{1} = 0$ and the SBP property $\bm{Q}_h + \bm{Q}_h^T = \bm{B}_h$, where $\bm{Q}_h = \bm{S}_h + \frac{1}{2}\bm{B}_h$.  Moreover, $\bm{Q}_h$ can be used to produce high order accurate approximations of derivatives, and can be used to construct entropy conservative schemes as in Theorem~\ref{thm:consentropy}.  The extension to a general number of elements is done in a similar fashion. }

It is possible to convert the semi-discrete entropy equality in Theorem~\ref{thm:consentropy} to a semi-discrete entropy inequality by adding appropriate interface dissipation terms, such as Lax-Friedrichs or matrix dissipation \cite{winters2017uniquely}.  We note that these terms must be computed in terms of $\tilde{\bm{u}}_f$ in order to ensure a discrete dissipation of entropy \cite{chen2017entropy, chan2017discretely}.  \note{Boundary conditions can also be incorporated into the formulation (\ref{eq:form}) in a weak fashion.  Let $\bm{u}_L, \bm{u}_R$ denote the values of the solution at the left and right domain boundaries of a 1D domain.  Using the SBP property, we can modify (\ref{eq:form}) to }
\[
\note{\bm{W}_h\td{\bm{u}_h}{t} + 2\bm{I}_h^T \LRp{\bm{S}_h\circ \bm{F}_S}\bm{1} + \bm{I}_h^T\bm{B}_h \begin{bmatrix}\bm{0}\\ \bm{f}(\bm{u}_L) \\ \vdots \\ \bm{0}\\ \bm{f}(\bm{u}_R)\end{bmatrix}}= 0,
\]
\note{where we have used the fact that $\bm{B}_h$ is diagonal and the consistency of $\bm{f}_S$ (implying that the diagonal of $\bm{F}_S$ reduces to the evaluation of the flux $\bm{f}(\bm{u})$) to simplify the boundary term $\LRp{\bm{B}_h\circ\bm{F}_S}\bm{1}$.  Boundary conditions are incorporated by replacing $\bm{f}(\bm{u}_L), \bm{f}(\bm{u}_R)$ with numerical fluxes $\bm{f}^*_L, \bm{f}^*_R$ as follows:}
\begin{equation}
\note{\bm{W}_h\td{\bm{u}_h}{t} + 2\bm{I}_h^T \LRp{\bm{S}_h\circ \bm{F}_S}\bm{1} + \bm{I}_h^T \bm{B}_h \begin{bmatrix}\bm{0}\\ \bm{f}^*_L \\ \vdots \\ \bm{0}\\ \bm{f}^*_R\end{bmatrix}}= 0.
\label{eq:bform}
\end{equation}
\note{If the boundary numerical flux satisfies the entropy stability conditions}
\[
\note{\psi_L - \bm{v}_L^T\bm{f}^*_{L} \leq 0, \qquad {\psi}_R - \bm{v}_R^T\bm{f}^*_{R} \leq 0}
\]
\note{then the resulting scheme satisfies a global entropy inequality \cite{chen2017entropy}.  }

\note{Assuming that the ODE system (\ref{eq:form}) exactly integrated in time and that entropy dissipative numerical fluxes and boundary conditions are used, the solution will satisfy a discrete entropy inequality (\ref{eq:weakentropyineq}).  In practice, the system (\ref{eq:form}) is solved using an ODE time-stepper.  For an explicit time-stepper, then all that is necessary is to invert the diagonal matrix $\bm{W}_h$ and evaluate the spatial terms in (\ref{eq:bform}).  

We note that the conservation or dissipation of entropy is guaranteed up to time-stepper accuracy.  In practice, we observe that entropy is dissipated for all problems considered, despite the fact that the proof does not hold at the fully discrete level.  A fully discrete entropy inequality can be guaranteed using implicit or space-time discretizations \cite{tadmor2003entropy, friedrich2018entropy}.}

\section{Extension to higher dimensions and non-affine meshes}
\label{sec:2}

The formulation in Theorem~\ref{thm:consentropy} can be naturally extended to Cartesian meshes in higher dimensions through a tensor product construction.  We first consider the construction of higher dimensional differentiation matrices on a two-dimensional reference element $\hat{\Omega}$, assuming a two dimensional tensor product grid of quadrature nodes (the construction of decoupled SBP operators in three dimensions is straightforward and similar to the two-dimensional case).  We then construct physical differentiation matrices on mapped elements $\Omega^k$, through which we construct an entropy conservative scheme.  

Let $\bm{D}_{\rm 1D}, \bm{W}_{\rm 1D}$ denote the 1D differentiation and mass matrices, respectively, on the reference interval $[-1,1]$.  Let ${\bm{W}}$ denote the 2D reference mass matrix, and let ${\bm{D}}^i$ denote the differentiation matrices with respect to the $i$th reference coordinate.  These matrices can be expressed in terms of Kronecker products  
\[
{\bm{D}}^1 \myeq \bm{D}_{\rm 1D} \otimes \bm{I}_{N+1}, \qquad {\bm{D}}^2  \myeq \bm{I}_{N+1} \otimes \bm{D}_{\rm 1D}, \qquad {\bm{W}} \myeq \bm{W}_{\rm 1D} \otimes  \bm{W}_{\rm 1D}, 
\]
where $\bm{I}_{N+1}$ denotes the $(N+1)\times (N+1)$ identity matrix.  

We also construct higher dimensional face interpolation matrices.  Let $\bm{V}_{f, {\rm 1D}}$ denote the one-dimensional interpolation matrix, and let $\bm{B}_{\rm 1D}$ denote the boundary matrix defined in Lemma~\ref{lemma:sbp}.  For an appropriate ordering of face quadrature points, the two-dimensional face interpolation matrix $\bm{V}_f$ and reference boundary matrices ${\bm{B}}^1, {\bm{B}}^2$ can be expressed as the concatenation of Kronecker product matrices
\[
\bm{V}_f \myeq \begin{bmatrix}
\bm{V}_{f, {\rm 1D}} \otimes \bm{I}_{2}\\
\bm{I}_{2} \otimes \bm{V}_{f, {\rm 1D}} 
\end{bmatrix}, \qquad 
{\bm{B}}^1 \myeq \begin{bmatrix}
\bm{B}_{\rm 1D} \otimes \bm{I}_{2} & \\
& \bm{0}
\end{bmatrix}, \qquad 
{\bm{B}}^2 \myeq \begin{bmatrix}
\bm{0} &\\
& \bm{I}_{2} \otimes \bm{B}_{\rm 1D} 
\end{bmatrix}.
\]
In three dimensions, the face interpolation matrix would be expressed as the concatenation of three Kronecker products involving $\bm{V}_{f, {\rm 1D}}$.  The higher dimensional differentiation and interpolation matrices $\bm{D}^i, \bm{V}_f$ can now be used to construct higher dimensional decoupled SBP operators.  Let ${\bm{Q}}^i_N$ denote the decoupled SBP operator for the $i$th coordinate on the reference element, where ${\bm{Q}}^i_N$ is defined as
\[
{\bm{Q}}^i_N \myeq \begin{bmatrix}
{\bm{Q}}^i - \frac{1}{2}\bm{V}_f^T{\bm{B}}^i\bm{V}_f & \frac{1}{2}\bm{V}_f^T{\bm{B}}^i\\
-\frac{1}{2}{\bm{B}}^i\bm{V}_f & \frac{1}{2}{\bm{B}}^i
\end{bmatrix}, \qquad {\bm{Q}}^i \myeq {\bm{W}}{\bm{D}}^i.
\]

Let the domain now be decomposed into non-overlapping elements $\Omega^k$, such that $\Omega^k$ is the image of $\hat{\Omega}$ under a degree $N$ polynomial mapping $\bm{\Phi}^k$.  We define geometric terms ${G}^k_{ij} = J^k\pd{\hat{x}_j}{\bm{x}_i}$ as scaled derivatives of reference coordinates $\hat{\bm{x}}$ with respect to physical coordinates $\bm{x}$.  We also introduce the scaled normals $n_iJ^k_f$, which can be computed on quadrilateral and tensor product elements via
\[
n_i J^k_f =\pm \sum_{j=1}^d G^k_{ij} = \pm \sum_{j=1}^d  J^k\pd{\hat{x}_j}{\bm{x}_i},
\]
where the sign of $n_i J^k_f$ is negative for a ``left'' face and positive for a ``right'' face.  
These geometric terms introduce scalings $J^k, J^k_f$, where $J^k$ is the determinant of the Jacobian of $\bm{\Phi}^k$ and $J^k_f$ denotes the determinant of the Jacobian of the mapping from a physical face to a reference face.  The quantities $G^k_{ij}, n_iJ^k_f$ can now be used to define matrices over each physical element $\Omega^k$ 
\begin{gather}
\bm{W}_k \myeq \bm{W} \diag{\bm{J}^k}, \qquad \bm{B}^i_k \myeq \bm{B}^i \diag{\bm{n}^k_i }, \nonumber\\
\bm{Q}^i_k \myeq \frac{1}{2} \sum_{j=1}^d \LRp{\diag{{\bm{G}}^k_{ij}}\bm{Q}^i_N  + \bm{Q}^i_N \diag{{\bm{G}}^k_{ij}} } \label{eq:splitcurv},
\end{gather}
where we have discretized the curved differentiation matrix in split form \cite{nordstrom2006conservative, kopriva2016geometry}.  
Here, $\bm{J}^k$ denotes the vector of values of $J^k$ at volume quadrature points, ${\bm{G}}^k_{ij}$ denotes the vector of values of ${G}^k_{ij}$ at volume and face quadrature points, and $\bm{n}^k_i$ denotes the $i$th scaled outward normal $n_i J^k_f$ at face quadrature points.  

A 2D Gauss collocation scheme can now be given in terms of $\bm{W}_k, \bm{Q}^i_k, \bm{B}^i_k$.   Let $\bm{u}^k_N$ denote the vector of solution values on $\Omega^k$, and define the quantities 
\begin{gather*}
\LRp{\bm{F}^i_S}_{mn} \myeq \bm{f}^i_S\LRp{\tilde{\bm{u}}_m,\tilde{\bm{u}}_n}, \qquad 1 \leq m,n \leq (N+1)^2 + 4(N+1) \\
\tilde{\bm{u}} \myeq \begin{bmatrix}
\bm{u}^k_N\\
\tilde{\bm{u}}^k_f
\end{bmatrix}, \qquad \tilde{\bm{u}}^k_f \myeq \bm{u}\LRp{\bm{v}^k_f}, \qquad 
\bm{v}^k_f \myeq {\bm{V}_f\bm{v}\LRp{\bm{u}^k_N}},
\end{gather*}
where $(N+1)^2 + 4(N+1)$ is the total number of volume and face points.  
We then have the following theorem on the semi-discrete conservation of entropy:
\begin{theorem}
\label{thm:esdg2d}
Assume that $\bm{Q}^k\bm{1} = 0$, and that $\Omega^k$ are mapped curvilinear elements.  Let $\bm{u}^k_N$ satisfy the local formulation on $\Omega^k$
\begin{equation}
{\bm{W}_k\td{\bm{u}^k_N}{t} + \sum_{i=1}^d \LRp{2\LRp{\bm{Q}^i_k\circ\bm{F}^i_S}\bm{1} + \bm{V}_f^T\bm{B}_k^i\LRp{\bm{f}^i_S\LRp{{\tilde{\bm{u}}_f}^+,\tilde{\bm{u}}^k_f}-\bm{f}^i\LRp{\tilde{\bm{u}}_f}}}} = 0, 
\end{equation}
where $\tilde{\bm{u}}^+$ denotes boundary values of the entropy-projected conservative variables on neighboring elements.  Then, $\bm{u}^k_N$ satisfies the discrete conservation of entropy 
\[
\sum_{k} \LRp{\bm{1}^T\bm{W}_k\td{S\LRp{\bm{u}^k_N}}{t} + \sum_{i=1}^d\bm{1}^T\bm{B}^i_k\LRp{\LRp{\bm{v}^k_f}^T\bm{f}^i(\tilde{\bm{u}}^k_f) - \psi\LRp{\tilde{\bm{u}}^k_f}}} = 0.
\]
\end{theorem}
The proof is a special case of the proof of Theorem 1 in \cite{chan2018discretely}, with the quadrature rule taken to be a tensor product rule with $(N+1)$ Gauss (or GLL) points in each coordinate direction.  

\begin{remark}
The operator $\bm{Q}^i_k$  in (\ref{eq:splitcurv}) can be applied without needing to explicitly store geometric terms $G^k_{ij}$ at face quadrature points.  By using the structure of the boundary matrix $\bm{B}^i$, expressions involving geometric terms on faces can be replaced by expressions involving components of scaled outward normals ${n}_i J^k_f$ on $\Omega^k$.  
\end{remark}

\subsection{Discrete geometric conservation law}

We note that Theorem~\ref{thm:esdg2d} relies on the assumption that $\bm{Q}^k\bm{1} = 0$.  This condition is equivalent to ensuring that the scaled geometric terms ${\bm{G}}^k_{ij}$ satisfy a discrete geometric conservation law (GCL) \cite{carpenter2014entropy, gassner2017br1, crean2018entropy, chan2018discretely}.  For two-dimensional degree $N$ (isoparametric) mappings, the GCL is automatically satisfied.  However, in three dimensions, the GCL is not guaranteed to be satisfied at the discrete level, due to the fact that geometric terms for isoparametric mappings are polynomials of degree higher than $N$.  

It is possible to ensure the satisfaction of a discrete GCL by using a sub-parametric polynomial geometric mapping.  Let $N_{\rm geo}$ denote the degree of a polynomial geometric mapping.  In three dimensions, the exact geometric terms for a degree $N_{\rm geo}$ polynomial mapping are polynomials of degree $2N_{\rm geo}$ \cite{kopriva2006metric, hindenlang2012explicit, crean2018entropy}.  If $2N_{\rm geo} \leq N$, or if $N_{\rm geo} \leq \left\lfloor\frac{N}{2}\right\rfloor$, then the discrete GCL is automatically satisfied.\footnote{We note that this condition is specific to three-dimensional hexahedral elements, and does not necessarily hold for other element types.  For example, the discrete GCL is satisfied if $2N_{\rm geo}-2 \leq N$ for isoparametric tetrahedral elements \cite{chan2018discretely}.}

For $N_{\rm geo} \geq \left\lfloor\frac{N}{2}\right\rfloor$, modifications to the computation of geometric terms are required to ensure that the GCL is satisfied at the discrete level.  For general SBP operators, the discrete GCL can be enforced through the solution of a local least squares problem \cite{crean2018entropy}.  We take an alternative approach and  construct geometric terms using the approach of Kopriva \cite{kopriva2006metric}.  This construction takes advantage of the fact that GLL and Gauss collocation methods correspond to polynomial discretizations.  Kopriva's construction is based on rewriting the geometric terms as the reference curl of an interpolated auxiliary quantity
\begin{equation}
\begin{bmatrix}
\bm{G}^k_{1j}\\
\bm{G}^k_{2j}\\
\bm{G}^k_{3j}
\end{bmatrix} = \begin{bmatrix}
\LRp{-\hat{\Grad} \times I_N\LRp{z\hat{\Grad}y}}_j\\
\LRp{\hat{\Grad} \times I_N\LRp{z\hat{\Grad}x}}_j\\
\LRp{\hat{\Grad} \times I_N\LRp{x\hat{\Grad}y}}_j
\end{bmatrix}, \qquad j = 1,2,3.
\label{eq:conscurl}
\end{equation}
Here, $I_N$ denotes the polynomial interpolation operator using GLL nodes.  By interpolating the auxiliary quantity in (\ref{eq:conscurl}) using polynomial interpolation prior to applying the curl, the geometric terms are approximated by degree $N$ polynomials which satisfy the discrete GCL by construction.  These GCL-satisfying geometric terms can then be used to compute normal vectors.  For a watertight mesh, the constructed normal vectors are guaranteed to be continuous across faces \cite{chan2018discretely}.  

To summarize, extending higher dimensional entropy stable Gauss collocation schemes to curved meshes requires the following steps: 
\begin{enumerate}
\item Construct polynomial approximations of the geometric terms using equation (\ref{eq:conscurl}) and interpolation at GLL nodes \cite{kopriva2006metric}.  
\item Evaluate approximate geometric terms at volume and surface points, and compute normal vectors in terms of the approximate geometric terms.  
\item Compute physical derivatives using the split form (\ref{eq:splitcurv}).  
\end{enumerate}
Apart from evaluating the GCL-satisfying geometric terms at separate volume and surface points, entropy stable Gauss collocation schemes are extended to curved meshes in the same manner as GLL and staggered-grid collocation schemes \cite{carpenter2014entropy, parsani2016entropy, chan2018discretely}.

\section{Numerical results}
\label{sec:3}

\note{In this section, we present numerical examples for} the compressible Euler equations, which are given in $d$ dimensions as follows:
\begin{align*}
\pd{\rho}{t} &+ \note{\sum_{j=1}^d \pd{\LRp{\rho {u}_j}}{x_j}} = 0,\\
\note{\pd{\rho {u}_i}{t}} &+ \note{\sum_{j=1}^d \pd{\LRp{\rho u_iu_j + p\delta_{ij} }}{x_j}} = 0, \qquad i = 1,\ldots,d\\
\pd{E}{t} &+ \note{\sum_{j=1}^d \pd{\LRp{{u}_j(E+p)}}{x_j}} = 0.\nonumber
\end{align*}
Here, $\rho$ is density, \note{$u_i$ denotes the $i$th component of} velocity, and $E$ is the total energy.  The pressure $p$ and specific internal energy $\rho e$ are given by 
\[
\note{p = (\gamma-1)\LRp{E - \frac{1}{2}\rho \LRp{\sum_{j=1}^d u_j^2}}}, \qquad \note{\rho e = E - \frac{1}{2}\rho \LRp{\sum_{j=1}^d u_j^2}}. 
\]

There exists an infinite family of suitable convex entropies for the compressible Euler equations \cite{harten1983symmetric}.  However, there is only a single unique entropy which appropriately treats the viscous heat conduction term in the compressible Navier-Stokes equations \cite{hughes1986new}.  This entropy $S(\bm{u})$ is given by
\begin{equation*}
S(\bm{u}) = -\frac{\rho s}{\gamma-1}, 
\end{equation*}
where $s = \log\LRp{\frac{p}{\rho^\gamma}}$ is the physical specific entropy, and the dimension $d = 1,2,3$.  The entropy variables in $d$ dimensions are given by
\begin{align*}
v_1 &= \frac{\rho e (\gamma + 1 - s) - E}{\rho e}, \\
v_{1+ i}&= \frac{\rho \note{{u}_i}}{\rho e}, \qquad i = 1,\ldots, d,\\
v_{d+2} &= -\frac{\rho}{\rho e},
\end{align*}
while the conservation variables in terms of the entropy variables are given by
\begin{align*}
\rho &= -(\rho e) v_{d+2}, \\
 \rho \note{u_i} &= (\rho e) v_{1+i}, \qquad i = 1,\ldots,d\\
 E &= (\rho e)\LRp{1 - \frac{\sum_{j=1}^d{v_{1+j}^2}}{2 v_{d+2}}},
\end{align*}
where the quantities $\rho e$ and $s$ in terms of the entropy variables are 
\begin{equation*}
\rho e = \LRp{\frac{(\gamma-1)}{\LRp{-v_{d+2}}^{\gamma}}}^{1/(\gamma-1)}e^{\frac{-s}{\gamma-1}}, \qquad s = \gamma - v_1 + \frac{\sum_{j=1}^d{v_{1+j}^2}}{2v_{d+2}}.
\end{equation*}
\note{Let $f$ denote some piecewise polynomial function, and let $f^+$ denote the exterior value of $f$ across an element face.  We define the average and logarithmic averages as follows:}
\[
\note{\avg{f} = \frac{f^+ + f}{2}}, \qquad \note{\avg{f}^{\log} = \frac{f^+ - f}{\log\LRp{f^+}-\log\LRp{f}}}.
\]
\note{The average and logarithmic average are assumed to act component-wise on vector-valued functions.  We evaluate the logarithmic average using the numerically stable algorithm of \cite{ismail2009affordable}.}
Explicit expressions for entropy conservative numerical fluxes in two dimensions  are given by Chandrashekar \cite{chandrashekar2013kinetic}
\begin{align*}
&f^1_{1,S}(\bm{u}_L,\bm{u}_R) = \avg{\rho}^{\log} \avg{u_1},& &f^1_{2,S}(\bm{u}_L,\bm{u}_R) = \avg{\rho}^{\log} \avg{u_2},&\\
&f^2_{1,S}(\bm{u}_L,\bm{u}_R) = f^1_{1,S} \avg{u_1} + p_{\rm avg},&  &f^2_{2,S}(\bm{u}_L,\bm{u}_R) = f^1_{2,S} \avg{u_1},&\nonumber\\
&f^3_{1,S}(\bm{u}_L,\bm{u}_R) = f^2_{2,S},& &f^3_{2,S}(\bm{u}_L,\bm{u}_R) = f^1_{2,S} \avg{u_2} + p_{\rm avg},&\nonumber\\
&f^4_{1,S}(\bm{u}_L,\bm{u}_R) = \LRp{E_{\rm avg} + p_{\rm avg}}\avg{u_1},& &f^4_{2,S}(\bm{u}_L,\bm{u}_R) = \LRp{E_{\rm avg} + p_{\rm avg} }\avg{u_2},& \nonumber
\end{align*}
where we have introduced the auxiliary quantities 
\begin{gather}
\note{\beta = \frac{\rho}{2p}}, \qquad p_{\rm avg} = \frac{\avg{\rho}}{2\avg{\beta}}, \qquad \qquad E_{\rm avg} = \frac{\avg{\rho}^{\log}}{2\avg{\beta}^{\log}\LRp{\gamma -1}}   + \frac{\note{u_{\rm avg}^2}}{2}, \label{eq:fluxaux} \\
\note{u^2_{\rm avg}} = 2(\avg{u_1}^2 + \avg{u_1}^2) - \LRp{\avg{u_1^2} +\avg{u_2^2}} \nonumber.  
\end{gather}
Expressions for entropy conservative numerical fluxes for the three-dimensional compressible Euler equations can also be explicitly written as
\begin{gather*}
\bm{f}_{1,S} = \LRp{\begin{array}{c}
\avg{\rho}^{\log}\avg{u_1}\\
\avg{\rho}^{\log}\avg{u_1}^2 + p_{\rm avg}\\
\avg{\rho}^{\log}\avg{u_1}\avg{u_2}\\
\avg{\rho}^{\log}\avg{u_1}\avg{u_3}\\
(E_{\rm avg}+ p_{\rm avg})\avg{u_1}\\
\end{array}}, 
\qquad 
\bm{f}_{2,S} = \LRp{\begin{array}{c}
\avg{\rho}^{\log}\avg{u_2}\\
\avg{\rho}^{\log}\avg{u_1}\avg{u_2}\\
\avg{\rho}^{\log}\avg{u_2}^2 + p_{\rm avg}\\
\avg{\rho}^{\log}\avg{u_2}\avg{u_3}\\
(E_{\rm avg}+ p_{\rm avg})\avg{u_2}\\
\end{array}},\\
\bm{f}_{3,S} = \LRp{\begin{array}{c}
\avg{\rho}^{\log}\avg{u_3}\\
\avg{\rho}^{\log}\avg{u_1}\avg{u_3}\\
\avg{\rho}^{\log}\avg{u_2}\avg{u_3}\\
\avg{\rho}^{\log}\avg{u_3}^2 + p_{\rm avg}\\
(E_{\rm avg}+ p_{\rm avg})\avg{u_3}\\
\end{array}},\nonumber
\end{gather*}
where the auxiliary quantities are defined as
\begin{gather*}
\note{\beta = \frac{\rho}{2p}}, \qquad p_{\rm avg} = \frac{\avg{\rho}}{2\avg{\beta}}, \qquad E_{\rm avg} = \frac{\avg{\rho}^{\log}}{2(\gamma-1)\avg{\beta}^{\log}} + \frac{1}{2}\avg{\rho}^{\log}\note{{u}^2_{\rm avg}}, \\
\note{u^2_{\rm avg}} = 2(\avg{u_1}^2 + \avg{u_2}^2 + \avg{u_3}^2) - \LRp{\avg{u_1^2} +\avg{u_2^2} + \avg{u_3^2}}.\nonumber
\end{gather*}

In all problems, \note{we use an explicit low storage RK-45 time-stepper \cite{carpenter1994fourth}} and estimate the timestep size $dt$ using $J, J^k_f$, and degree-dependent $L^2$ trace constants $C_N$ 
\[
dt = C_{\rm CFL}\frac{ h}{ a C_N}, \qquad h = \frac{1}{{\nor{J^{-1}}_{L^{\infty}}}\nor{J^k_f}_{L^{\infty}}},
\]
where $a$ is an estimate of the maximum wave speed, $h$ estimates the mesh size, and $C_{\rm CFL}$ is some user-defined CFL constant.  For isotropic elements, the ratio of $J^k$ to $J^k_f$ scales as the mesh size $h$, while $C_N$ captures the dependence of the maximum stable timestep on the polynomial degree $N$.  For hexahedral elements, $C_N$ varies depending on the choice of quadrature.  It was shown in \cite{chan2015gpu} that 
\[
C_N =\begin{cases}
 d\frac{N(N+1)}{2} & \text{for GLL nodes}\\
d\frac{(N+1)(N+2)}{2} & \text{for Gauss nodes}
\end{cases}.
\]
Thus, based on this rough estimate of the maximum stable timestep, GLL collocation schemes should be able to take a timestep which is roughly $(1 + 2/N)$ times larger than the maximum stable timestep for Gauss collocation schemes.  We do not account for this discrepancy in this work, and instead set the timestep for both GLL and Gauss collocation schemes based on the more conservative Gauss collocation estimate of $dt$.  

Numerical results in 1D are similar to those presented in \cite{chan2017discretely}.  Thus, we focus on two and three dimensional problems and comparisons of entropy stable GLL and Gauss collocation schemes.  

\subsection{2D isentropic vortex problem}

We begin by examining high order convergence of the proposed methods in two dimensions using the isentropic vortex problem \cite{shu1998essentially, crean2017high}.  The analytical solution is 
\begin{align}
\rho(\bm{x},t) &= \LRp{1 - \frac{\frac{1}{2}(\gamma-1)(\beta e^{1-r(\bm{x},t)^2})^2}{8\gamma \pi^2}}^{\frac{1}{\gamma-1}}, \qquad p = \rho^{\gamma},\\
u_1(\bm{x},t) &= 1 - \frac{\beta}{2\pi} e^{1-r(\bm{x},t)^2}(y-y_0), \qquad u_2(\bm{x},t) = \frac{\beta}{2\pi} e^{1-r(\bm{x},t)^2}(x-x_0-t),\nonumber
\end{align}
where $u_1, u_2$ are the $x$ and $y$ velocity and $r(\bm{x},t) = \sqrt{(x-x_0-t)^2 + (y-y_0)^2}$.  Here, we take $x_0 = 5, y_0 = 0$ and $\beta = 5$.  

\begin{figure}
\centering
\subfloat[Lightly warped mesh]{\includegraphics[width=.475\textwidth]{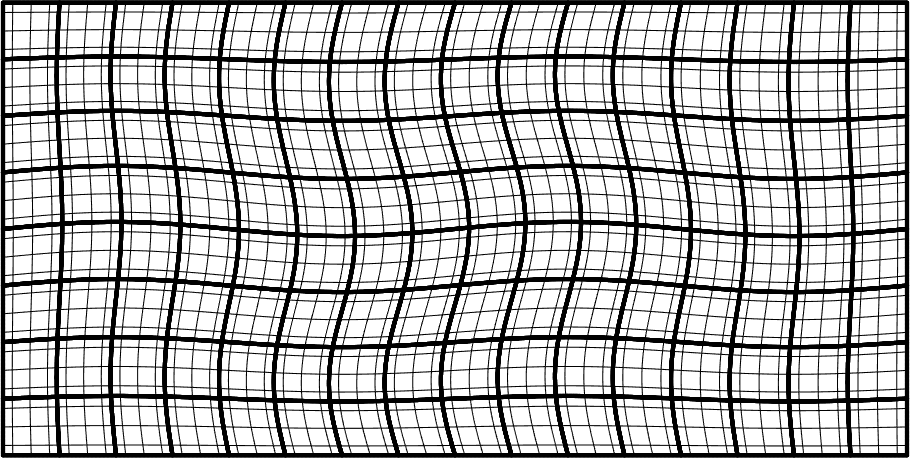}}
\hspace{1em}
\subfloat[Moderately warped mesh]{\includegraphics[width=.475\textwidth]{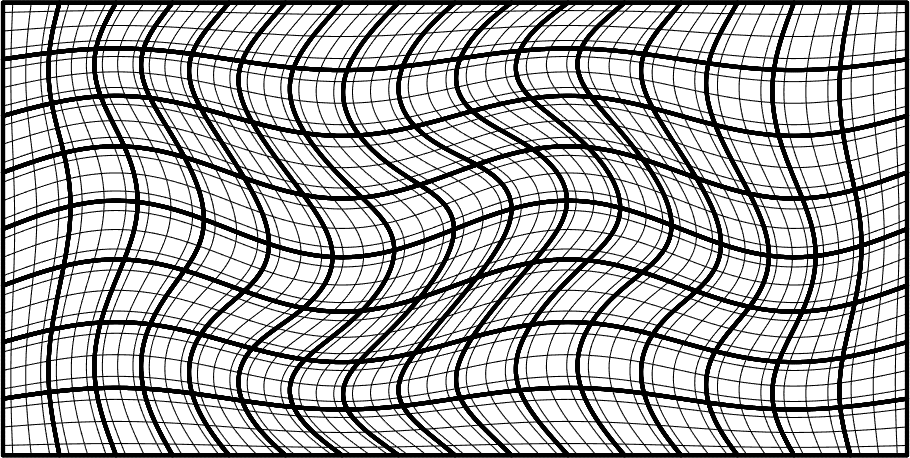}}\\
\subfloat[Heavily warped mesh]{\includegraphics[width=.45\textwidth]{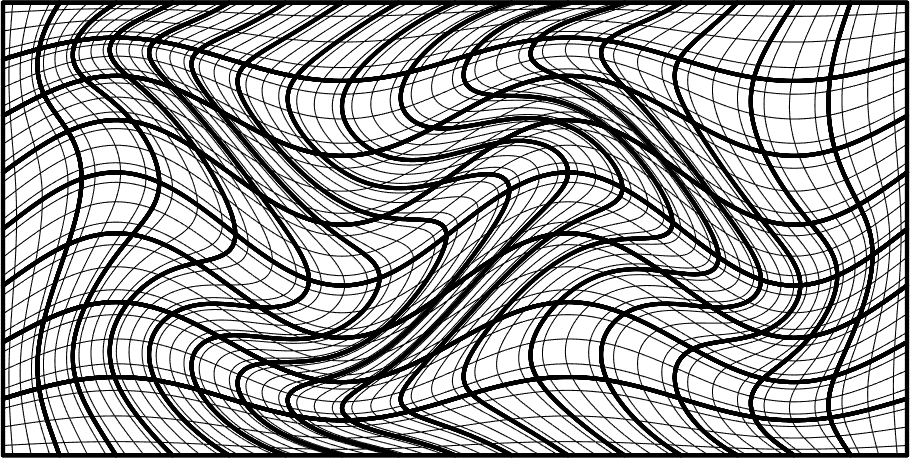}}
\caption{Lightly, moderately, and heavily warped meshes for $N=4, K= 16$.  }
\label{fig:warp2d}
\end{figure}

We solve on a periodic rectangular domain $[0, 20] \times [-5,5]$ until final time $T=5$, and compute errors over all all solution fields.  For a degree $N$ approximation, we approximate the $L^2$ error using an $(N+2)$ point Gauss quadrature rule.  We also examine the influence of element curvature for both GLL and Gauss collocation schemes by examining $L^2$ errors on a sequence of moderately and heavily warped curvilinear meshes (see Figure~\ref{fig:warp2d}).  These warpings are constructed by modifying nodal positions according to the following mapping
\note{
\begin{align*}
\tilde{x} &= x + L_x\alpha\cos\LRp{\frac{\pi}{L_x}\LRp{x-\frac{L_x}{2}}}\cos\LRp{\frac{3\pi}{L_y}y}, \\
\tilde{y} &= y + L_y\alpha\sin\LRp{\frac{4\pi}{L_x}\LRp{\tilde{x}-\frac{L_x}{2}}}\cos\LRp{\frac{\pi}{L_y}y},
\end{align*}
where $L_x = 20, L_y = 10$ denote the lengths of the domain in the $x$ and $y$ directions, respectively.  
The lightly warped mesh corresponds to $\alpha = 1/64$, the moderately warped mesh corresponds to $\alpha = 1/16$, and the heavily warped mesh corresponds to $\alpha = 1/8$.  All results are computed using $C_{\rm CFL} = 1/2$ and Lax-Friedrichs interface dissipation.  }

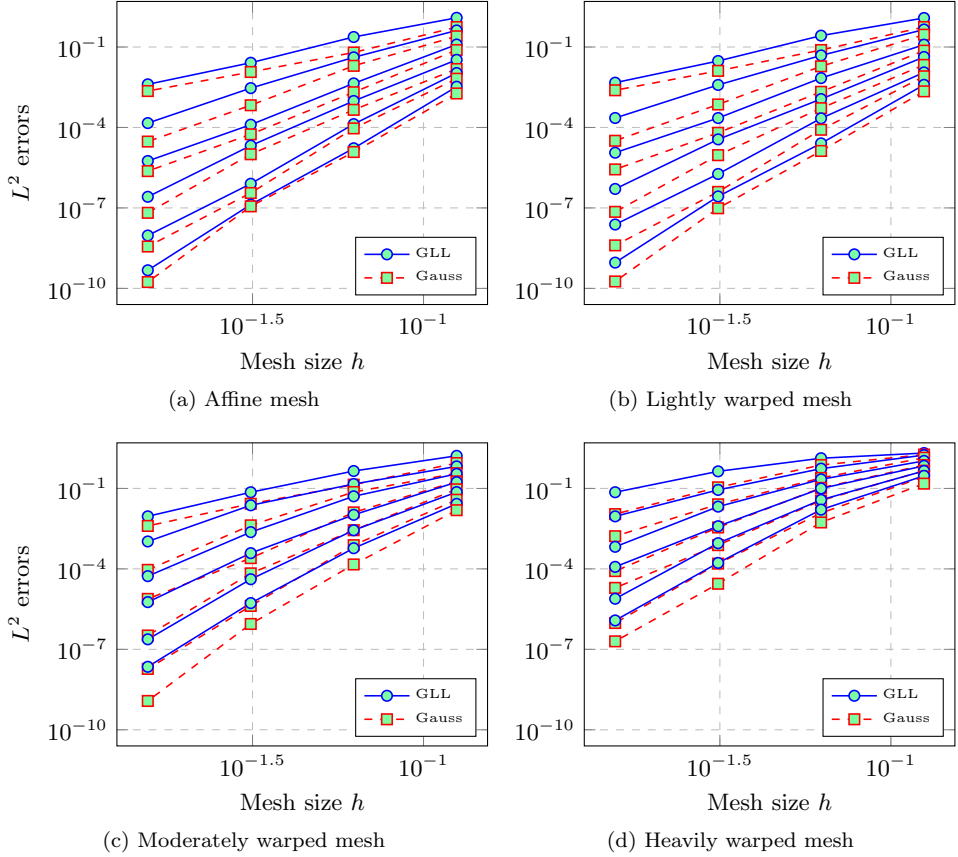
\begin{figure}
\centering
\subfloat[Affine mesh]{
\begin{tikzpicture}
\begin{loglogaxis}[
    width=.5\textwidth,
    xlabel={Mesh size $h$},
    ylabel={$L^2$ errors}, 
    ymin=2.5e-11, ymax=5,
    legend pos=south east, legend cell align=left, legend style={font=\tiny},	
    xmajorgrids=true, ymajorgrids=true, grid style=dashed,
    legend entries={GLL, Gauss}    
]
\pgfplotsset{
cycle list={{blue, mark=*}, {red, dashed ,mark=square*}}
}
\addplot+[semithick, mark options={solid, fill=markercolor}]
coordinates{(0.125,1.22861)(0.0625,0.236646)(0.03125,0.0258883)(0.015625,0.00406122)};
\addplot+[semithick, mark options={solid, fill=markercolor}]
coordinates{(0.125,0.553435)(0.0625,0.0625441)(0.03125,0.0117156)(0.015625,0.0022841)}; 

\addplot+[semithick, mark options={solid, fill=markercolor}]
coordinates{(0.125,0.414449)(0.0625,0.0412224)(0.03125,0.00293039)(0.015625,0.000145701)};
\addplot+[semithick, mark options={solid, fill=markercolor}]
coordinates{(0.125,0.25082)(0.0625,0.0198863)(0.03125,0.000673152)(0.015625,2.98263e-05)}; 

\addplot+[semithick, mark options={solid, fill=markercolor}]
coordinates{(0.125,0.123516)(0.0625,0.00439285)(0.03125,0.000126852)(0.015625,5.61786e-06)};
\addplot+[semithick, mark options={solid, fill=markercolor}]
coordinates{(0.125,0.0768018)(0.0625,0.0020579)(0.03125,5.55322e-05)(0.015625,2.36108e-06)}; 

\addplot+[semithick, mark options={solid, fill=markercolor}]
coordinates{(0.125,0.0332509)(0.0625,0.000970148)(0.03125,2.10334e-05)(0.015625,2.60008e-07)};
\addplot+[semithick, mark options={solid, fill=markercolor}]
coordinates{(0.125,0.0151231)(0.0625,0.000454595)(0.03125,1.00203e-05)(0.015625,6.57796e-08)}; 

\addplot+[semithick, mark options={solid, fill=markercolor}]
coordinates{(0.125,0.0106752)(0.0625,0.000133284)(0.03125,7.97032e-07)(0.015625,9.35152e-09)};
\addplot+[semithick, mark options={solid, fill=markercolor}]
coordinates{(0.125,0.00657096)(0.0625,9.22845e-05)(0.03125,3.60025e-07)(0.015625,3.63974e-09)}; 

\addplot+[semithick, mark options={solid, fill=markercolor}]
coordinates{(0.125,0.00330917)(0.0625,1.67825e-05)(0.03125,1.32696e-07)(0.015625,4.73518e-10)};
\addplot+[semithick, mark options={solid, fill=markercolor}]
coordinates{(0.125,0.00185273)(0.0625,1.21465e-05)(0.03125,1.14289e-07)(0.015625,1.72358e-10)}; 
\end{loglogaxis}
\end{tikzpicture}
}
\subfloat[Lightly warped mesh]{
\begin{tikzpicture}
\begin{loglogaxis}[
    width=.5\textwidth,
    xlabel={Mesh size $h$},
    ymin=2.5e-11, ymax=5,
      legend pos=south east, legend cell align=left, legend style={font=\tiny},	
    xmajorgrids=true, ymajorgrids=true, grid style=dashed,
    legend entries={GLL, Gauss}    
]
\pgfplotsset{
cycle list={{blue, mark=*}, {red, dashed ,mark=square*}}
}
\addplot+[semithick, mark options={solid, fill=markercolor}]
coordinates{(0.125,1.2082)(0.0625,0.265357)(0.03125,0.0301118)(0.015625,0.00467693)};
\addplot+[semithick, mark options={solid, fill=markercolor}]
coordinates{(0.125,0.555586)(0.0625,0.0773534)(0.03125,0.0129171)(0.015625,0.0024363)};
\addplot+[semithick, mark options={solid, fill=markercolor}]
coordinates{(0.125,0.448878)(0.0625,0.0485323)(0.03125,0.00385418)(0.015625,0.000225396)};
\addplot+[semithick, mark options={solid, fill=markercolor}]
coordinates{(0.125,0.287685)(0.0625,0.0190431)(0.03125,0.000721055)(0.015625,3.19095e-05)};
\addplot+[semithick, mark options={solid, fill=markercolor}]
coordinates{(0.125,0.122081)(0.0625,0.00694351)(0.03125,0.000224524)(0.015625,1.14555e-05)};
\addplot+[semithick, mark options={solid, fill=markercolor}]
coordinates{(0.125,0.0731583)(0.0625,0.00213914)(0.03125,6.32128e-05)(0.015625,2.73037e-06)};
\addplot+[semithick, mark options={solid, fill=markercolor}]
coordinates{(0.125,0.0424351)(0.0625,0.00114569)(0.03125,3.59697e-05)(0.015625,5.09643e-07)};
\addplot+[semithick, mark options={solid, fill=markercolor}]
coordinates{(0.125,0.0215298)(0.0625,0.000523714)(0.03125,9.27521e-06)(0.015625,7.14031e-08)};
\addplot+[semithick, mark options={solid, fill=markercolor}]
coordinates{(0.125,0.0112838)(0.0625,0.000219049)(0.03125,1.84304e-06)(0.015625,2.40402e-08)};
\addplot+[semithick, mark options={solid, fill=markercolor}]
coordinates{(0.125,0.00809138)(0.0625,8.23415e-05)(0.03125,3.95565e-07)(0.015625,4.01268e-09)};
\addplot+[semithick, mark options={solid, fill=markercolor}]
coordinates{(0.125,0.00388316)(0.0625,2.55074e-05)(0.03125,2.68102e-07)(0.015625,9.13754e-10)};
\addplot+[semithick, mark options={solid, fill=markercolor}]
coordinates{(0.125,0.00222029)(0.0625,1.33111e-05)(0.03125,9.72405e-08)(0.015625,1.81386e-10)};

\end{loglogaxis}
\end{tikzpicture}
}
\\
\subfloat[Moderately warped mesh]{
\begin{tikzpicture}
\begin{loglogaxis}[
    width=.5\textwidth,
    xlabel={Mesh size $h$},
    ylabel={$L^2$ errors}, 
    ymin=2.5e-11, ymax=5,
    legend pos=south east, legend cell align=left, legend style={font=\tiny},	
    xmajorgrids=true, ymajorgrids=true, grid style=dashed,
    legend entries={GLL, Gauss}    
]
\pgfplotsset{
cycle list={{blue, mark=*}, {red, dashed ,mark=square*}}
}
\addplot+[semithick, mark options={solid, fill=markercolor}]
coordinates{(0.125,1.64274)(0.0625,0.445945)(0.03125,0.0720148)(0.015625,0.0091066)};
\addplot+[semithick, mark options={solid, fill=markercolor}]
coordinates{(0.125,0.882247)(0.0625,0.138116)(0.03125,0.0268145)(0.015625,0.0039945)};

\addplot+[semithick, mark options={solid, fill=markercolor}]
coordinates{(0.125,0.662165)(0.0625,0.146131)(0.03125,0.0234934)(0.015625,0.00105368)};
\addplot+[semithick, mark options={solid, fill=markercolor}]
coordinates{(0.125,0.368084)(0.0625,0.0734059)(0.03125,0.00420438)(0.015625,9.19378e-05)};

\addplot+[semithick, mark options={solid, fill=markercolor}]
coordinates{(0.125,0.349184)(0.0625,0.0508268)(0.03125,0.00234785)(0.015625,5.36192e-05)};
\addplot+[semithick, mark options={solid, fill=markercolor}]
coordinates{(0.125,0.183472)(0.0625,0.0127617)(0.03125,0.000251929)(0.015625,7.6013e-06)};

\addplot+[semithick, mark options={solid, fill=markercolor}]
coordinates{(0.125,0.171986)(0.0625,0.0103578)(0.03125,0.000386834)(0.015625,5.76284e-06)};
\addplot+[semithick, mark options={solid, fill=markercolor}]
coordinates{(0.125,0.0924215)(0.0625,0.00276519)(0.03125,6.80528e-05)(0.015625,3.27246e-07)};

\addplot+[semithick, mark options={solid, fill=markercolor}]
coordinates{(0.125,0.0729963)(0.0625,0.00276104)(0.03125,4.11305e-05)(0.015625,2.36125e-07)};
\addplot+[semithick, mark options={solid, fill=markercolor}]
coordinates{(0.125,0.0378013)(0.0625,0.0007649)(0.03125,4.15848e-06)(0.015625,1.8685e-08)};

\addplot+[semithick, mark options={solid, fill=markercolor}]
coordinates{(0.125,0.026577)(0.0625,0.000589552)(0.03125,5.3039e-06)(0.015625,2.23719e-08)};
\addplot+[semithick, mark options={solid, fill=markercolor}]
coordinates{(0.125,0.0154563)(0.0625,0.000146207)(0.03125,8.83156e-07)(0.015625,1.18829e-09)};

\end{loglogaxis}
\end{tikzpicture}
}
\subfloat[Heavily warped mesh]{
\begin{tikzpicture}
\begin{loglogaxis}[
    width=.5\textwidth,
    xlabel={Mesh size $h$},
    ymin=2.5e-11, ymax=5,
    legend pos=south east, legend cell align=left, legend style={font=\tiny},	
    xmajorgrids=true, ymajorgrids=true, grid style=dashed,
    legend entries={GLL, Gauss}    
]
\pgfplotsset{
cycle list={{blue, mark=*}, {red, dashed ,mark=square*}}
}
\addplot+[semithick, mark options={solid, fill=markercolor}]
coordinates{(0.125,2.08578)(0.0625,1.33668)(0.03125,0.433653)(0.015625,0.0721002)};
\addplot+[semithick, mark options={solid, fill=markercolor}]
coordinates{(0.125,1.82833)(0.0625,0.746283)(0.03125,0.11006)(0.015625,0.0110115)};
\addplot+[semithick, mark options={solid, fill=markercolor}]
coordinates{(0.125,1.74649)(0.0625,0.546097)(0.03125,0.0879841)(0.015625,0.0090851)};
\addplot+[semithick, mark options={solid, fill=markercolor}]
coordinates{(0.125,1.39268)(0.0625,0.245852)(0.03125,0.0255562)(0.015625,0.00165226)};
\addplot+[semithick, mark options={solid, fill=markercolor}]
coordinates{(0.125,1.06289)(0.0625,0.218361)(0.03125,0.0209471)(0.015625,0.000665425)};
\addplot+[semithick, mark options={solid, fill=markercolor}]
coordinates{(0.125,0.740071)(0.0625,0.105075)(0.03125,0.003505)(0.015625,8.26421e-05)};
\addplot+[semithick, mark options={solid, fill=markercolor}]
coordinates{(0.125,0.695283)(0.0625,0.0996971)(0.03125,0.00390099)(0.015625,0.000118992)};
\addplot+[semithick, mark options={solid, fill=markercolor}]
coordinates{(0.125,0.457324)(0.0625,0.0354479)(0.03125,0.000767746)(0.015625,1.95823e-05)};
\addplot+[semithick, mark options={solid, fill=markercolor}]
coordinates{(0.125,0.449493)(0.0625,0.0374955)(0.03125,0.000909834)(0.015625,7.70064e-06)};
\addplot+[semithick, mark options={solid, fill=markercolor}]
coordinates{(0.125,0.274721)(0.0625,0.0120499)(0.03125,0.000157175)(0.015625,9.61243e-07)};
\addplot+[semithick, mark options={solid, fill=markercolor}]
coordinates{(0.125,0.294238)(0.0625,0.0162153)(0.03125,0.000166829)(.015625,1.19661e-06)};
\addplot+[semithick, mark options={solid, fill=markercolor}]
coordinates{(0.125,0.150916)(0.0625,0.00538908)(0.03125,2.79182e-05)(.015625,1.98845e-07)};
\end{loglogaxis}
\end{tikzpicture}
}
\caption{$L^2$ errors for the 2D isentropic vortex at time $T=5$ for degree $N = 2,\ldots,7$ GLL and Gauss collocation schemes.}
\label{fig:err2d}
\end{figure}

Figure~\ref{fig:err2d} shows the $L^2$ errors for affine, moderately warped, and heavily warped meshes.  For affine meshes, Gauss collocation results in a lower errors than GLL collocation at all orders.  However, the difference between both schemes decreases as $N$ increases.  This is not too surprising: on a Cartesian domain, the discrete $L^2$ inner product resulting from GLL quadrature converges to exact $L^2$ inner product over the space of polynomials as $N$ increases \cite{quarteroni1994introduction}.  However, GLL and Gauss collocation differ more significantly on curved meshes.  For both moderately and heavily warped meshes, the errors for a degree $N$ Gauss collocation scheme are nearly identical to errors for a higher order GLL collocation scheme of degree $(N+1)$.  These results are in line with numerical experiments in \cite{parsani2016entropy}, which show that GLL collocation schemes lose one order of convergence in the $L^2$ norm on unstructured non-uniform meshes.  Both results show that increasing quadrature accuracy significantly reduces the effect of polynomial aliasing due to curved meshes and spatially varying geometric terms.  

We note that $L^2$ approximation estimates on curved meshes \cite{lenoir1986optimal, warburton2013low} assume that the mesh size is small enough to be in the asymptotic regime (such that asymptotic error estimates hold).  On curved meshes, this requires that the mesh is sufficiently fine to resolve both the solution and geometric mapping.   The results in Figure~\ref{fig:err2d} suggest that, compared to Gauss collocation schemes, under-integrated GLL collocation schemes require a finer mesh resolution to reach the asymptotic regime.  

\subsection{3D isentropic vortex problem}

As in two dimensions, we test the accuracy of the proposed scheme using an isentropic vortex solution adapted to three dimensions.  The solution is the extruded 2D vortex propagating in the $y$ direction, with an analytic expression given in \cite{williams2013nodal}
\begin{align*}
\rho(\bm{x},t) &= \LRp{1-\frac{(\gamma-1)}{2}\Pi^2}^{\frac{1}{\gamma-1}}\\
u_i(\bm{x},t) &= \Pi r_i, \\
E(\bm{x},t) &= \frac{p_0}{\gamma-1}\LRp{1-\frac{\gamma-1}{2}\Pi^2}^{\frac{\gamma}{\gamma-1}} + \frac{\rho}{2}\sum_{i=1}^d u_i^2.
\end{align*}
where $(u_1,u_2,u_3)^T$ is the three-dimensional velocity vector and 
\[
\Pi = \Pi_{\max}e^{\frac{1-\sum_{i=1}^dr_i^2}{2}}, \qquad \begin{pmatrix}r_1\\r_2\\r_3\end{pmatrix} = \begin{pmatrix}
-(x_2-c_2-t)\\
x_1-c_1\\
0
\end{pmatrix}.
\]
We take \note{$c_1 = c_2 = 7.5$}, $p_0 = {1}/{\gamma}$, and $\Pi_{\max} = 0.4$, and solve on the domain \note{$[0,15]\times [0,20]\times [0,5]$ until final time $T=5$.  We decompose the domain into uniform hexahedral elements with edge length $h$}.  As in the 2D case, we also examine the effect of curvilinear mesh warping.  We construct a curved warping of the initial Cartesian mesh by mapping nodes on each hexahedron to warped nodal positions $(\tilde{x},\tilde{y},\tilde{z})$ through the transformation
\begin{align*}
\tilde{y} &= y + \note{\frac{1}{8}L_y\cos\LRp{3\pi \frac{(x-7.5)}{15}}\cos\LRp{\pi \frac{(y-10)}{20}}\cos\LRp{\pi \frac{(z-2.5)}{5}}}\\
\tilde{x} &= x + \note{\frac{1}{8}L_x\cos\LRp{\pi \frac{(x-7.5)}{15}}\sin\LRp{4\pi \frac{(\tilde{y}-10)}{20}}\cos\LRp{\pi \frac{(z-2.5)}{5}}}\\
\tilde{z} &= z + \note{\frac{1}{8}L_z\cos\LRp{\pi \frac{(\tilde{x}-7.5)}{15}}\cos\LRp{2\pi \frac{(\tilde{y}-10)}{20}}\cos\LRp{\pi \frac{(z-2.5)}{5}}}.
\end{align*} 
where $L_x = 15, L_y = 20, L_z = 5$.  
On curved meshes, the geometric terms are constructed using the approach of Kopriva described in Section~\ref{sec:2}.  A $C_{\rm CFL} = .75$ is used for all experiments.  

Figure~\ref{fig:err3d} shows the $L^2$ errors for degrees \note{$N = 2,\ldots, 5$}.\footnote{For the $N = 2$ GLL collocation scheme on the coarsest 3D mesh, the error is not shown because the solution diverged. } As in the 2D case, Gauss collocation schemes produce smaller errors than GLL collocation.  The difference between the two schemes on affine meshes is slightly more pronounced than in 2D, while the difference between the two schemes on curved meshes is less significant than observed in 2D experiments.  \note{We expect that this difference is due to the specific curved mapping.  The warped 2D mesh used in Figure~\ref{fig:err2d} was generated to mimic a severe ``vorticular'' warping.  The 3D mapping is less severely warped, due to different  domain dimensions and difficulties ensuring invertibility of the map from the reference to physical element.}

\begin{figure}
\centering
\subfloat[Affine mesh]{
\begin{tikzpicture}
\begin{loglogaxis}[
    width=.5\textwidth,
    xlabel={Mesh size $h$},
    ylabel={$L^2$ errors}, 
    ymin=1e-7, ymax=5,
    legend pos=south east, legend cell align=left, legend style={font=\tiny},	
    xmajorgrids=true, ymajorgrids=true, grid style=dashed,
    legend entries={GLL, Gauss}    
]
\pgfplotsset{
cycle list={{blue, mark=*}, {red, dashed ,mark=square*}}
}
\addplot+[semithick, mark options={solid, fill=markercolor}]
coordinates{(2.5,1.85364)(1.25,0.514785)(0.625,0.0793524)(0.3125,0.0132308)};
\addplot+[semithick, mark options={solid, fill=markercolor}]
coordinates{(2.5,1.08061)(1.25,0.164679)(0.625,0.0294895)(0.3125,0.00570425)};

\addplot+[semithick, mark options={solid, fill=markercolor}]
coordinates{(2.5,0.661158)(1.25,0.135728)(0.625,0.00847132)(0.3125,0.000394629)};
\addplot+[semithick, mark options={solid, fill=markercolor}]
coordinates{(2.5,0.397063)(1.25,0.0610903)(0.625,0.00189633)(0.3125,9.49625e-05)};

\addplot+[semithick, mark options={solid, fill=markercolor}]
coordinates{(2.5,0.325704)(1.25,0.0172973)(0.625,0.000580552)(0.3125,2.81526e-05)};
\addplot+[semithick, mark options={solid, fill=markercolor}]
coordinates{(2.5,0.169055)(1.25,0.00547459)(0.625,0.000154222)(0.3125,7.74137e-06)};

\addplot+[semithick, mark options={solid, fill=markercolor}]
coordinates{(2.5,0.0977112)(1.25,0.00526152)(0.625,8.73468e-05)(0.3125,9.44473e-07)};
\addplot+[semithick, mark options={solid, fill=markercolor}]
coordinates{(2.5,0.0386913)(1.25,0.00171584)(0.625,2.46858e-05)(0.3125,2.07682e-07)};

\end{loglogaxis}
\end{tikzpicture}
}
\subfloat[Curved mesh]{
\begin{tikzpicture}
\begin{loglogaxis}[
    width=.5\textwidth,
    xlabel={Mesh size $h$},
    ylabel={$L^2$ errors}, 
    ymin=1e-7, ymax=5,
    legend pos=south east, legend cell align=left, legend style={font=\tiny},	
    xmajorgrids=true, ymajorgrids=true, grid style=dashed,
    legend entries={GLL, Gauss}    
]
\pgfplotsset{
cycle list={{blue, mark=*}, {red, dashed ,mark=square*}}
}
\addplot+[semithick, mark options={solid, fill=markercolor}]
coordinates{(2.5,NaN)(1.25,1.57425)(0.625,0.448088)(0.3125,0.0871033)};
\addplot+[semithick, mark options={solid, fill=markercolor}]
coordinates{(2.5,2.48256)(1.25,0.981474)(0.625,0.232946)(0.3125,0.0529084)};

\addplot+[semithick, mark options={solid, fill=markercolor}]
coordinates{(2.5,2.03576)(1.25,0.599779)(0.625,0.101389)(0.3125,0.0120874)};
\addplot+[semithick, mark options={solid, fill=markercolor}]
coordinates{(2.5,1.62319)(1.25,0.336149)(0.625,0.0485175)(0.3125,0.00371689)};

\addplot+[semithick, mark options={solid, fill=markercolor}]
coordinates{(2.5,1.22642)(1.25,0.240606)(0.625,0.0221903)(0.3125,0.000696802)};
\addplot+[semithick, mark options={solid, fill=markercolor}]
coordinates{(2.5,0.90428)(1.25,0.131914)(0.625,0.00712465)(0.3125,0.000207085)};

\addplot+[semithick, mark options={solid, fill=markercolor}]
coordinates{(2.5,0.720447)(1.25,0.0986167)(0.625,0.00367923)(0.3125,0.000120058)};
\addplot+[semithick, mark options={solid, fill=markercolor}]
coordinates{(2.5,0.524023)(1.25,0.0500179)(0.625,0.00121588)(0.3125,3.23497e-05)};
\end{loglogaxis}
\end{tikzpicture}
}
\caption{$L^2$ errors for the 3D isentropic vortex for \note{$N = 2,\ldots,5$} on sequences of Cartesian and curved meshes.}
\label{fig:err3d}
\end{figure}
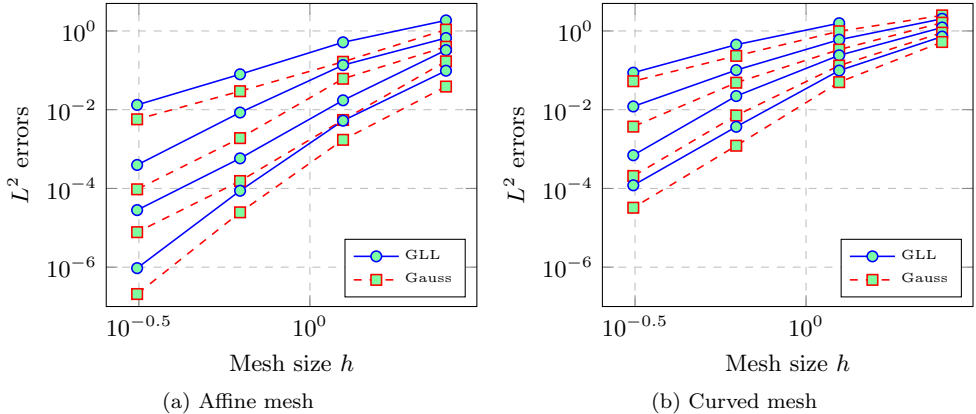

We also compared $L^2$ errors for both isoparametric and sub-parametric geometric mappings.  In both cases, the discrete GCL is satisfied.  For sub-parametric mappings, we chose the degree of approximation of geometry $N_{\rm geo} = \left\lfloor\frac{N}{2} \right\rfloor + 1$, such that the geometric terms are computed exactly and the GCL is satisfied by default.   This test is intended to address the fact that the GCL-preserving interpolation of Kopriva introduces a small approximation error, as the geometric terms are no longer exact.  For these sub-parametric mappings, the gap between GLL and Gauss collocation widens slightly at $N=2$.  However, the results for sub-parametric mappings are nearly identical to the results in the isoparametric case for higher polynomial degrees.

\subsection{Shock-vortex interaction}

The next problem considered is the shock-vortex interaction described in \cite{shu1998essentially}.  The domain is taken to be $[0,2]\times [0,1]$, and is triangulated with uniform quadrilateral elements.  Wall boundary conditions are imposed on the top and bottom boundaries, and inflow boundary conditions are typically imposed on the left and right boundaries.  We modify the problem setup such that periodic boundary conditions are imposed at the left and right boundaries.  Wall boundary conditions are imposed using a mirror state for the normal velocity, which was shown to be entropy stable in \cite{svard2014entropy, chen2017entropy}.  

The initial condition is taken to be the superposition of a stationary shock and a vortex propagating towards the right.  The stationary Mach $M_s = 1.1$ shock is positioned at $x = .5$ normal to the $x$ axis, with left state $\LRp{\rho_L, u_L, v_L, p_L} = \LRp{1, \sqrt{\gamma}, 0, 1}$, \note{where $u_L, v_L$ denote $x,y$ components of the left-state velocity}.  The right state is a scaling of the left state computed using the Rankine-Hugoniot conditions, such that the ratio of upstream and downstream states is
\begin{align*}
\frac{\rho_L}{\rho_R} = \frac{u_L}{u_R} = \frac{2+ (\gamma-1) M_s^2}{(\gamma+1)M_s^2}, \qquad \frac{p_L}{p_R} = 1+ \frac{2\gamma}{\gamma+1}\LRp{M_s^2-1}, \qquad v_R = 0.
\end{align*}
The isentropic vortex is centered at $(x_c,y_c) = (.25, .5)$ and given in terms of velocity fluctuations $\delta u$ and $\delta v$, which are functions of the tangential velocity $v_{\theta}$
\[
\delta u = v_{\theta} \sin(\theta), \qquad
\delta v = -v_{\theta} \cos(\theta), \qquad
v_{\theta} = \epsilon \tau e^{\alpha(1-\tau^2)},
\]
where $r = \sqrt{(x-x_c)^2 + (y-y_c)^2}$ is the radius from the vortex center, $\tau = r/r_c$, and $\theta = \tan^{-1}\LRp{\frac{y-y_c}{x-x_c}}$.  We follow \cite{shu1998essentially} and take $\epsilon = .3$, $\alpha = .204$, and $r_c = .05$.  The vortex temperature is computed as a fluctuation $\delta T$ of the upstream state $T_L = p_L/\rho_L$
\[
\delta T =  - \frac{(\gamma-1)\epsilon^2 e^{2\alpha (1-\tau^2)}}{4\alpha\gamma}.  
\]
The vortex density and pressure are computed using an isentropic assumption.  To summarize, the initial condition for the shock-vortex interaction problem is 
\begin{gather*}
\rho = \rho_s \LRp{\frac{T_{\rm vor}}{T_L}}^{\frac{1}{\gamma-1}}, \qquad u_1 = u_s + \delta u, \qquad u_2 = v_s + \delta v, \qquad p = p_s \LRp{\frac{T_{\rm vor}}{T_L}}^{\frac{1}{\gamma-1}},
\end{gather*}
where $(\rho_s,  u_s, v_s, p_s)$ denote the discontinuous stationary shock solution given by the left and right states $\LRp{\rho_L, u_L, v_L, p_L}, \LRp{\rho_R, u_R, v_R, p_R}$.  

\begin{figure}
\centering
\subfloat[Entropy conservative flux, $T = .3$]{\includegraphics[width=.495\textwidth]{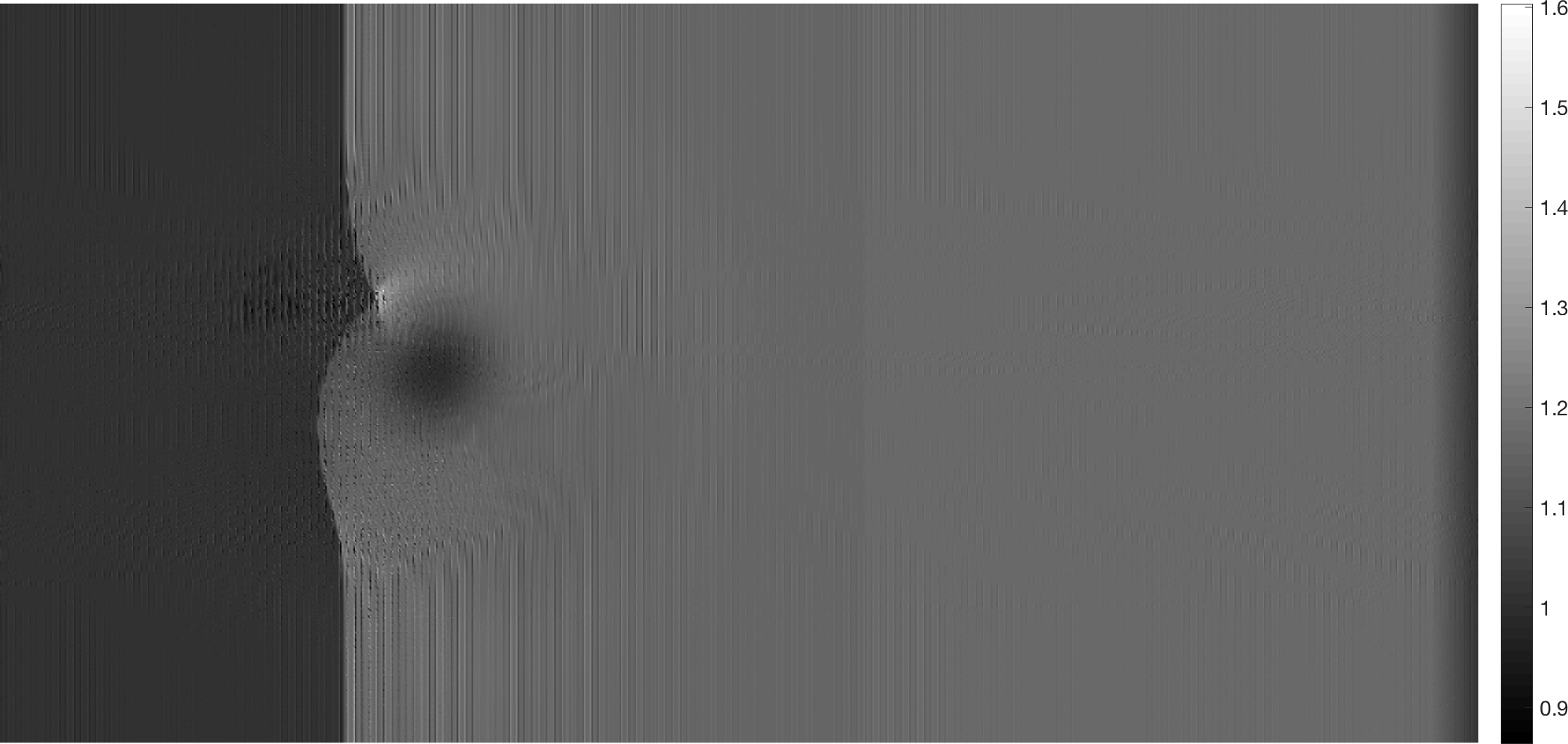}}
\hspace{.05em}
\subfloat[Entropy conservative flux, $T = .7$]{\includegraphics[width=.495\textwidth]{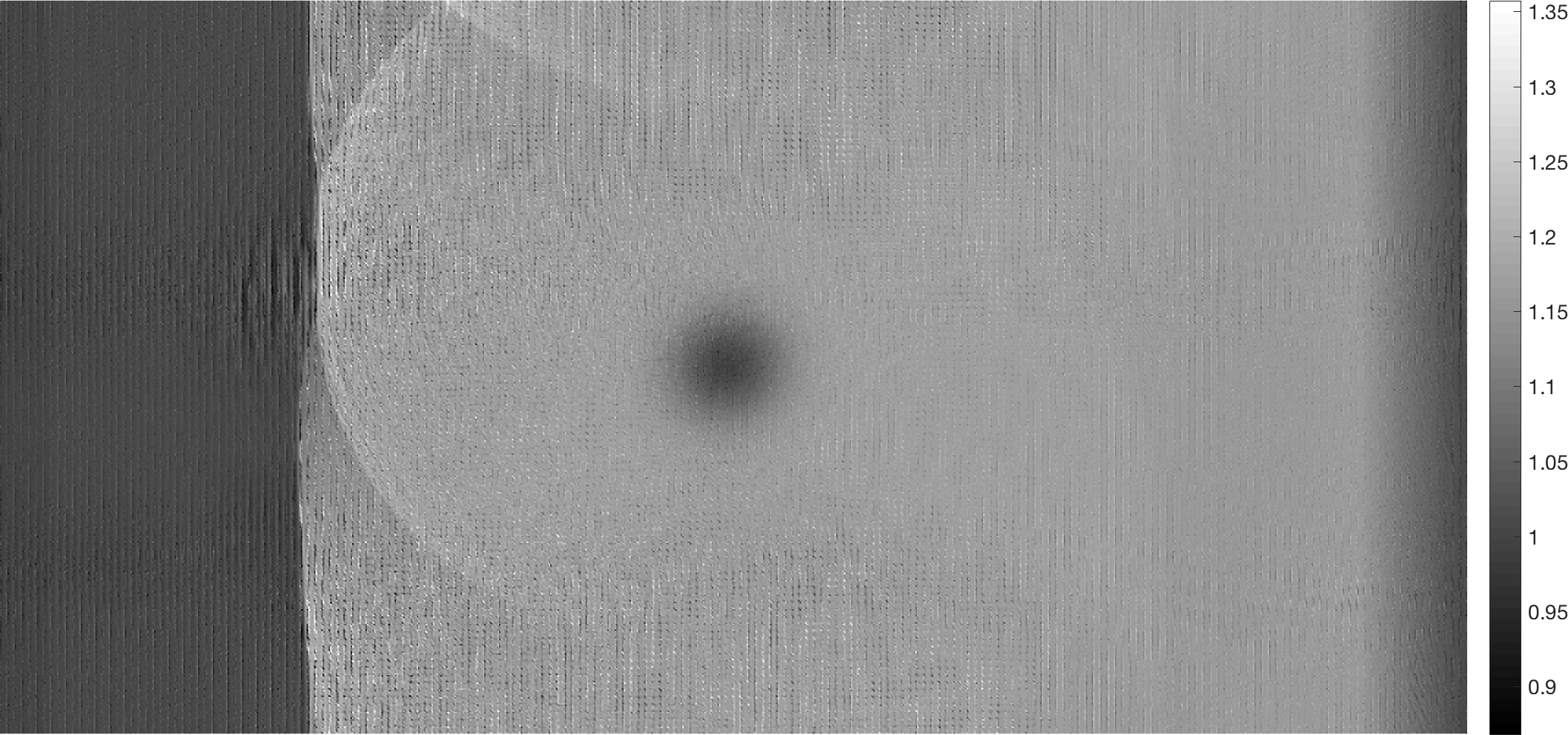}}\\
\subfloat[Lax-Friedrichs flux, $T = .3$]{\includegraphics[width=.495\textwidth]{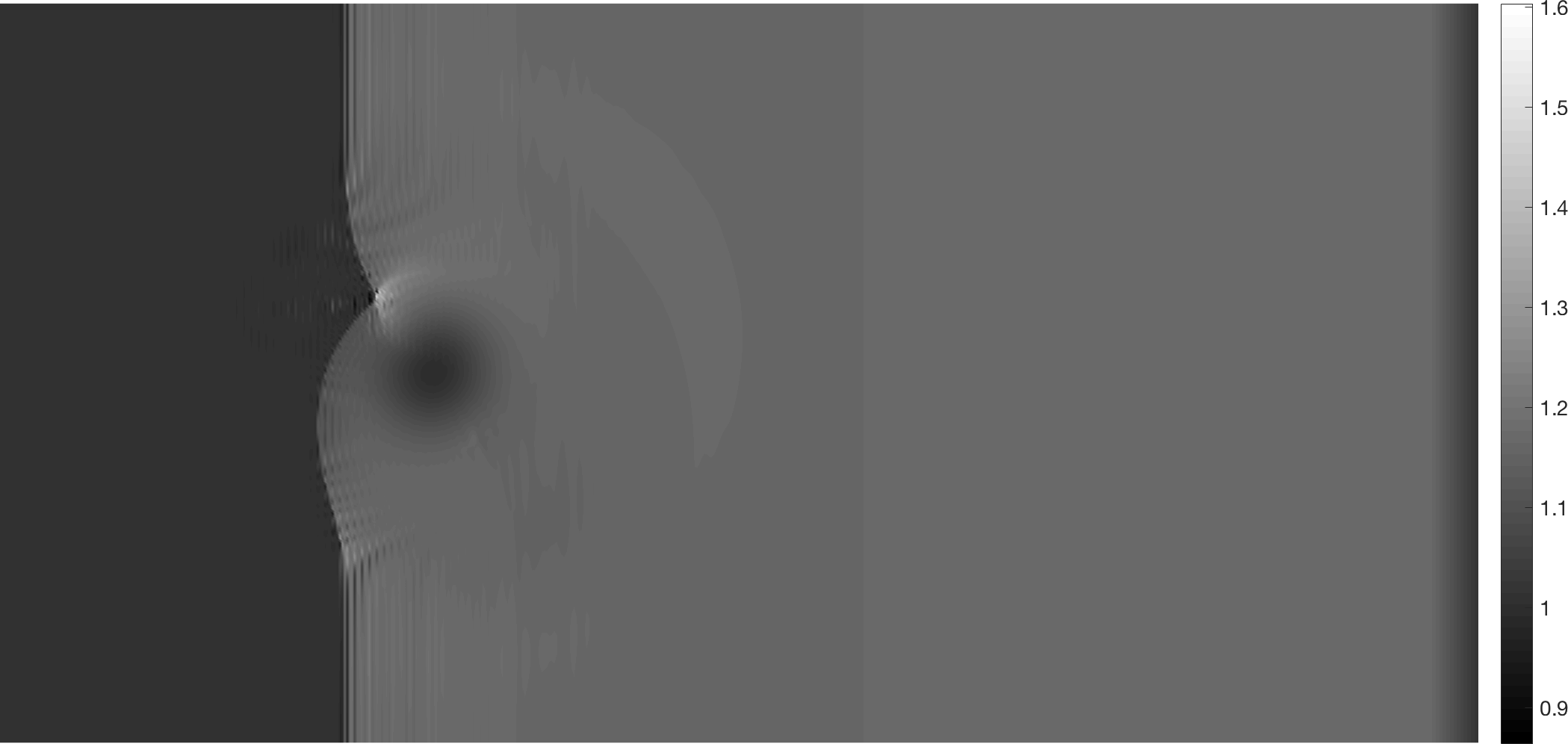}}
\hspace{.05em}
\subfloat[Lax-Friedrichs flux, $T = .7$]{\includegraphics[width=.495\textwidth]{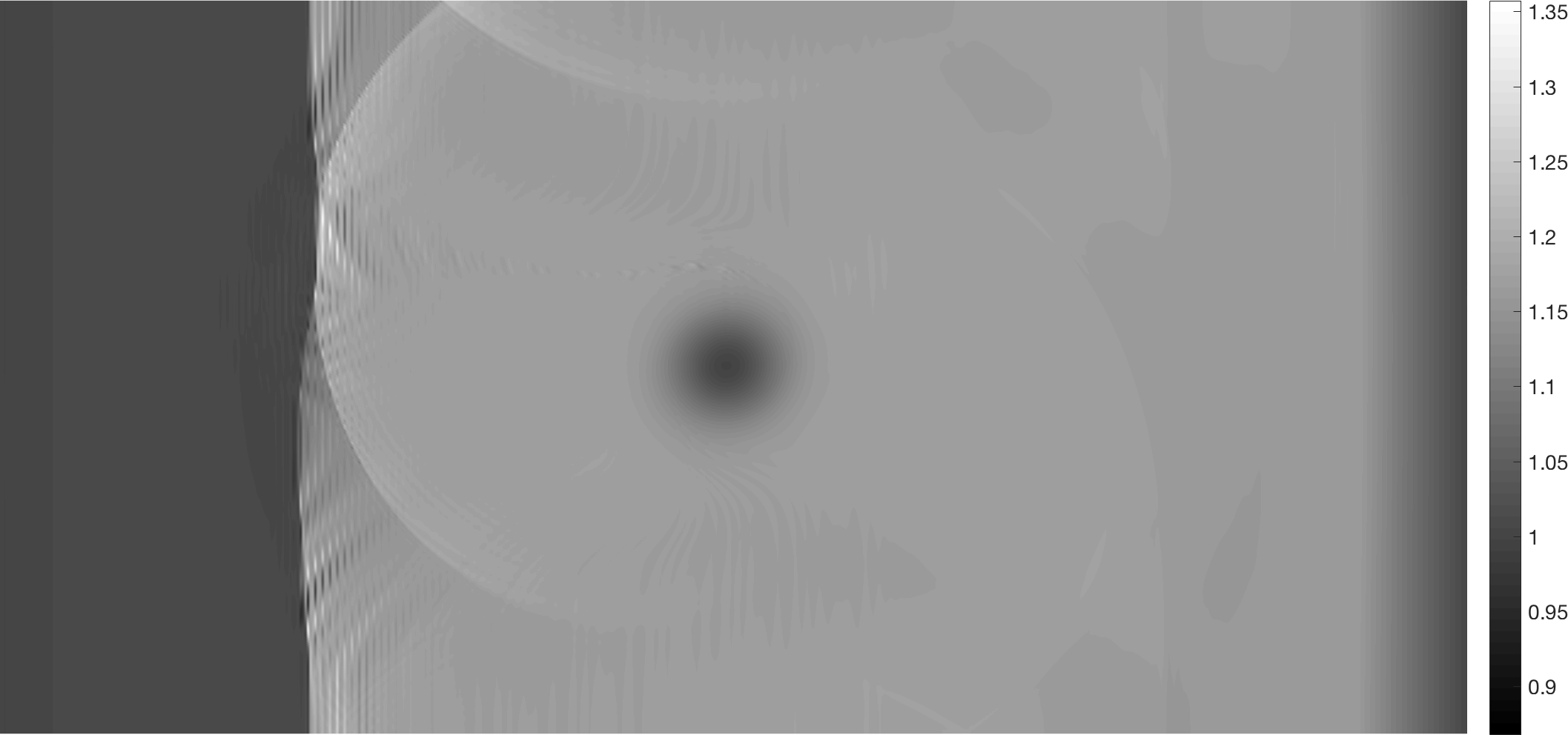}}\\
\subfloat[Matrix dissipation flux, $T = .3$]{\includegraphics[width=.495\textwidth]{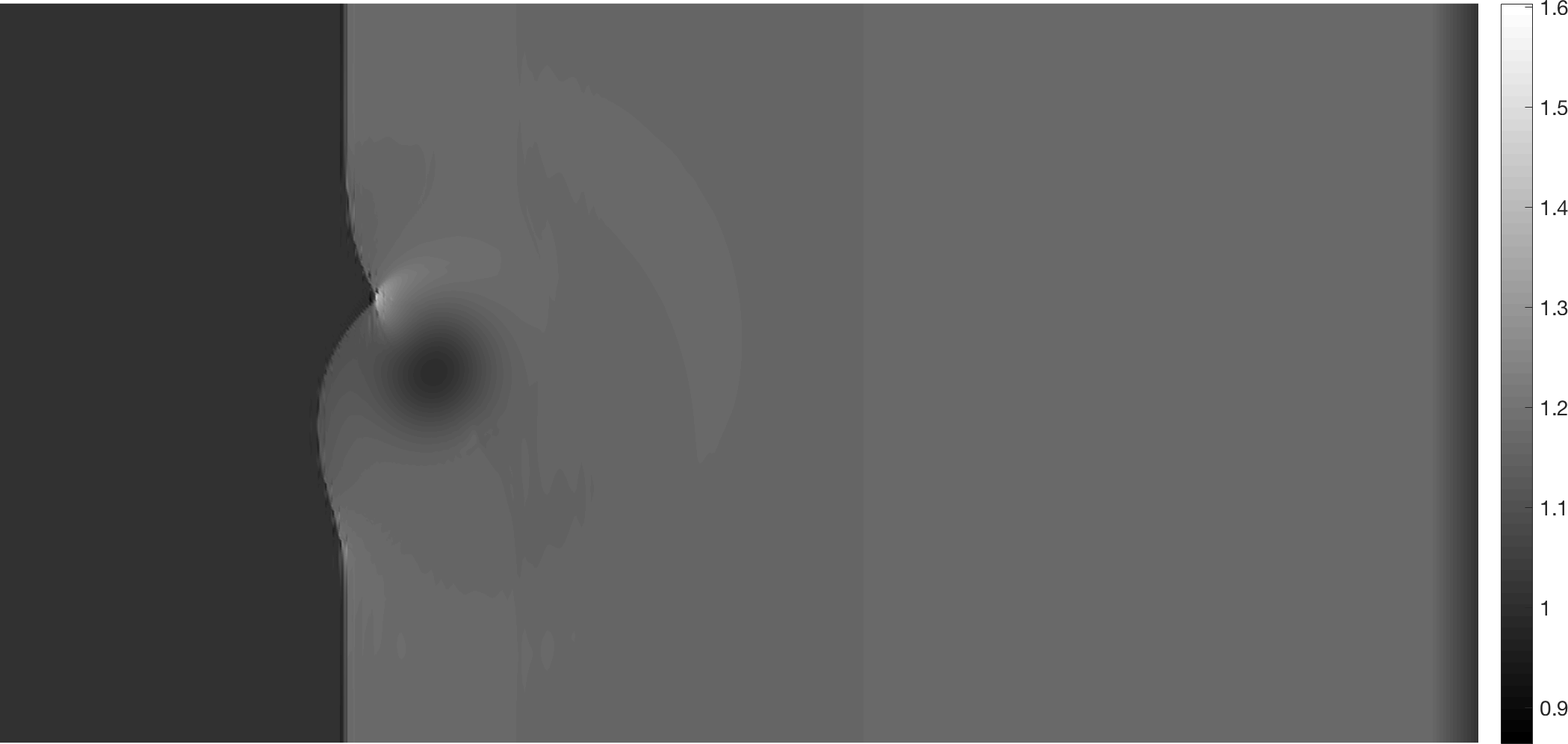}}
\hspace{.05em}
\subfloat[Matrix dissipation flux, $T = .7$]{\includegraphics[width=.495\textwidth]{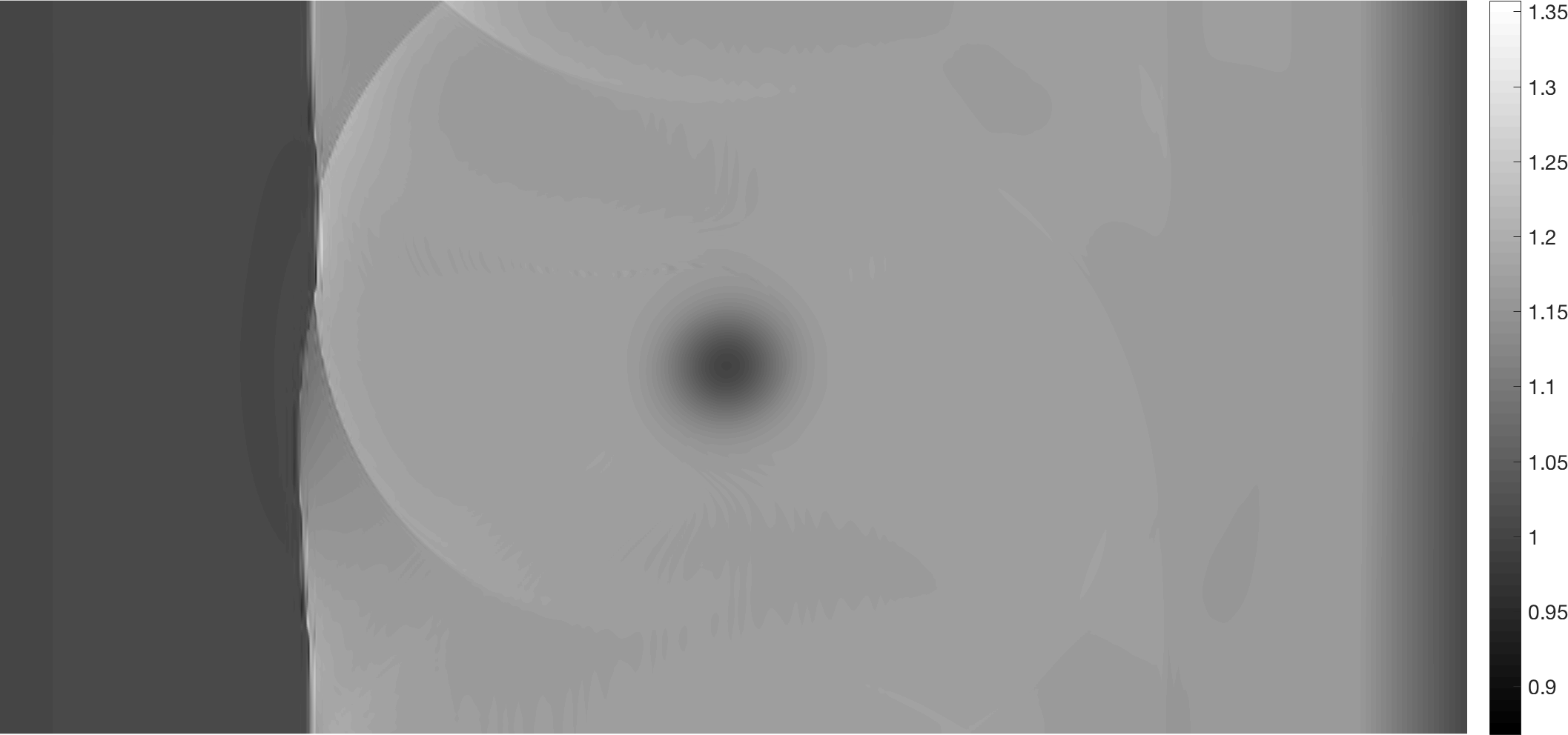}}
\caption{Shock vortex solution at time $T = .7$ using entropy stable Gauss collocation schemes with $N=4, h = 1/100$.}
\label{fig:shockvort}
\end{figure}

We compare three different entropy stable Gauss collocation schemes.  All three utilize the entropy conservative flux of Chandrashekar \cite{chandrashekar2013kinetic}.  For the first scheme, we do not introduce any additional interface dissipation, which produces an entropy conservative scheme.  The second scheme introduces an entropy-dissipative interface term using Lax-Friedrichs penalization, while the third scheme utilizes the matrix dissipation flux introduced in \cite{winters2017uniquely}.  This flux adds a dissipation of the form $\bm{R}\bm{D}\bm{R}^T\jump{{\bm{v}}^k_f}$, where $\jump{{\bm{v}}^k_f}$ denotes the jump in the entropy variables.  In two dimensions, the matrices $\bm{R}, \bm{D}$ are 
\begin{align*}
\bm{R} &= \begin{bmatrix}
1 & 1 & 0 & 1\\
\avg{u_1}-\bar{a} n_x & \avg{u_1} & n_y & \avg{u_1}+\bar{a} n_x\\
\avg{u_2}-\bar{a} n_y & \avg{u_2} & -n_x & \avg{u_2}+\bar{a} n_y\\
h-\bar{a} \bar{u}_n & \frac{1}{2}u^2_{\rm avg} & \bar{u}_n n_y - \avg{u_2}n_x & h + \bar{a} \bar{u}_n
\end{bmatrix}\\
\bm{D} &= {\rm diag}{\begin{pmatrix}
\LRb{\bar{u}_n - a}\frac{\avg{\rho}^{\log}}{2\gamma}, & \LRb{\bar{u}_n}\frac{\avg{\rho}^{\log}(\gamma-1)}{\gamma}, & \LRb{\bar{u}_n}p_{\rm avg}, & \LRb{\bar{u}_n + a}\frac{\avg{\rho}^{\log}}{2\gamma} \end{pmatrix}},
\end{align*}
where $\bar{a}, \bar{u}_n$ are defined in 2D as
\begin{gather*}
\bar{u}_n = \avg{u_1}n_x + \avg{u_2}n_y, \qquad \bar{a} = \sqrt{\frac{\gamma p_{\rm avg}}{\avg{\rho}^{\log}}}, \qquad h = \frac{\gamma}{2(\gamma-1)\avg{\beta}^{\log}} + \frac{1}{2}\bar{u}, 
\end{gather*}
and $p_{\rm avg}, u^2_{\rm avg}$ are defined as in (\ref{eq:fluxaux}).  

Figure~\ref{fig:shockvort} shows density solutions for $N=4$ and $h = 1/100$ Gauss collocation schemes using a non-dissipative entropy conservative flux, a dissipative Lax-Friedrichs flux, and a matrix dissipation flux.  In all cases, the vortex passes through the shock stably without the use of additional slope limiting, filtering, or artificial viscosity.  However, the entropy conservative scheme produces a large number of spurious oscillations in the solution.  These are reduced away from the shock under Lax-Friedrichs dissipation, though oscillations still persist around a large neighborhood of the discontinuity.  The Gibbs-type oscillations are most localized under the matrix dissipation flux.  

We note that this experiment also verifies that entropy stable decoupled SBP schemes (including the over-integrated case \cite{chan2017discretely}) are compatible with entropy stable wall boundary conditions.  As far as the authors know, the stable and high order accurate imposition of such boundary conditions for existing GSBP couplings described in \cite{crean2017high} and Section~\ref{sec:gsbpsat} remains an open problem.  

\subsection{Inviscid Taylor-Green vortex}

We conclude by investigating the behavior of entropy stable Gauss collocation schemes for the inviscid Taylor--Green vortex \cite{ae1937mechanism, gassner2016split, crean2018entropy}.  This problem is posed on the periodic box $[-\pi,\pi]^3$, with initial conditions 
\begin{gather*}
\rho = 1, \qquad p = \frac{100}{\gamma} + \frac{1}{16} \LRp{\cos(2x_1) + \cos(2x_2)}\LRp{2+\cos(2x_3)},\\
\note{u_1 = \sin(x_1)\cos(x_2)\cos(x_3),} \qquad
\note{u_2 = -\cos(x_1)\sin(x_2)\cos(x_3),}\qquad
\note{u_3 = 0.}
\end{gather*}
The Taylor--Green vortex is used to study the transition and decay of turbulence.  In the absence of viscosity, the Taylor--Green vortex develops successively smaller scales as time increases.  As a result, the solution is guaranteed to contain under-resolved features after a sufficiently large time.  We study the evolution of kinetic energy $\kappa(t)$ 
\[
\kappa(t) =\frac{1}{\LRb{\Omega}} \int_{\Omega} \rho \note{ \LRp{\sum_{i=1}^d u_i^2}} \diff{\bm{x}},
\]
as well as the kinetic energy dissipation rate $-\pd{\kappa}{t}$, which is approximated by differencing $\kappa(t)$.  For both GLL and Gauss collocation schemes, integrals in the kinetic energy formula are evaluated using an $(N+1)$-point Gauss quadrature rule.

\begin{figure}
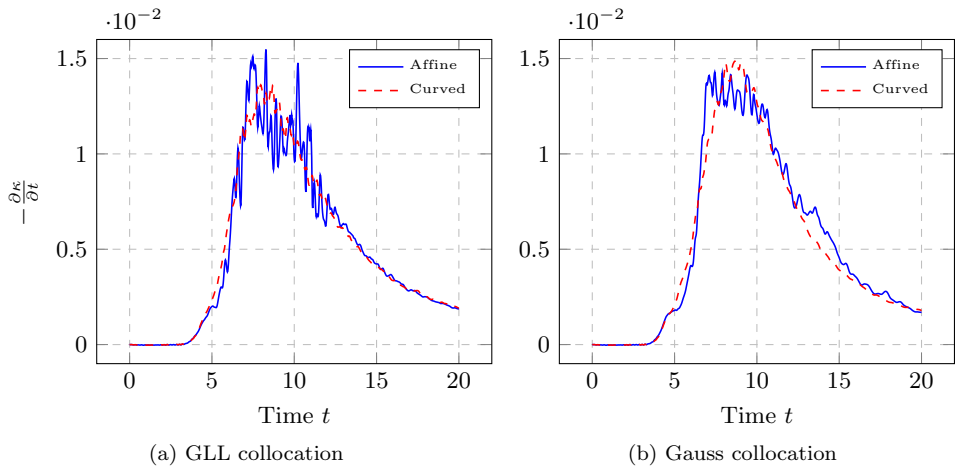

\centering
\subfloat[GLL collocation]{

}
\caption{Kinetic energy dissipation rate for entropy stable GLL and Gauss collocation schemes with $N=7$ and $h = \pi/8$. }
\label{fig:tgv}
\end{figure}

Figure~\ref{fig:tgv} shows the evolution of the kinetic energy dissipation rate from $t\in [0,20]$ for Gauss and GLL collocation schemes on affine and curved meshes.  The curved meshes used here are constructed by modifying nodal positions through the mapping
\[
\tilde{\bm{x}} = \bm{x} + \frac{1}{2}\sin(x)\sin(y)\sin(z).
\]
All cases utilize $N=7$ and $h = \pi/8$ (corresponding to $8$ elements per side), as well as a CFL of $.25$.  Lax-Friedrichs dissipation is used for all simulations.  For both affine and curved meshes, the presented Gauss collocation schemes remain stable in the presence of highly under-resolved solution features.  Kinetic energy dissipation rates for both GLL and Gauss collocation are qualitatively similar and are consistent with existing results in the literature for the inviscid Taylor-Green vortex \cite{gassner2016split,chan2018discretely}.

\section{A theoretical cost comparison}

While the numerical experiments presented in previous sections demonstrate several advantages of Gauss collocation methods over GLL collocation schemes, these do not account for additional costs associated with Gauss collocation schemes.  While a detailed time-to-solution comparison is outside of the scope of this work, we can compare computational costs associated with Gauss, staggered-grid, and GLL collocation schemes.  

The main computational costs associated with entropy stable schemes are volume operations, which include evaluations of two-point fluxes and applications of one-dimensional differentiation and interpolation matrices.  We do not count inter-element communication or flux computations, as these are typically less expensive than volume operations (especially for higher polynomial degrees).  The total number of flux computations for each scheme can also be reduced using the symmetry of $\bm{f}_S$; however, this does not affect relative differences in the number of two-point flux evaluations between GLL, Gauss, and staggered grid schemes.  However, exploiting symmetry in the Gauss scheme is slightly more straightforward (compared to GLL and staggered grid schemes) due to the block structure of the decoupled SBP operator.

In 3D, a degree $N$ GLL collocation scheme contains $(N+1)^3$ nodes.  Two-point fluxes are computed between states at one node and states at $(N+1)$ additional nodes, resulting in a total of $3(N+1)^4$ two-point flux evaluations.  For a staggered grid scheme, two-point fluxes are computed on a degree $(N+1)$ GLL grid consisting of $(N+2)^3$ nodes, resulting in $3(N+2)^4$ evaluations.  

Gauss collocation schemes compute two-point fluxes on a grid of Gauss nodes, resulting in $3(N+1)^4$ flux evaluations in 3D.  One must also evaluate two-point fluxes between face nodes and volume nodes, and vice versa.  As a result, Gauss schemes evaluate two-point fluxes twice between each face node and the line of $(N+1)$ volume nodes normal to that face, resulting in an extra $12(N+1)^3$ flux evaluations in 3D.  

We also consider costs associated with applying operator matrices.  GLL and Gauss collocation schemes require applying a one-dimensional differentiation matrix to each line of nodes, resulting in $O\LRp{3(N+1)^4}$ operations in 3D.  Staggered-grid schemes require $O\LRp{3(N+2)^4}$ operations (corresponding to differentiation in each coordinate on a degree $(N+1)$ GLL grid) as well as $O\LRp{6(N+2)(N+1)^3}$ operations required for interpolation to and from a $(N+1)$ point Gauss grid to a $(N+2)$ point GLL grid in three dimensions.  Gauss collocation schemes require only interpolation to and from points on 6 faces, resulting in an additional $O\LRp{12(N+1)^3}$ operations per dimension per element.  

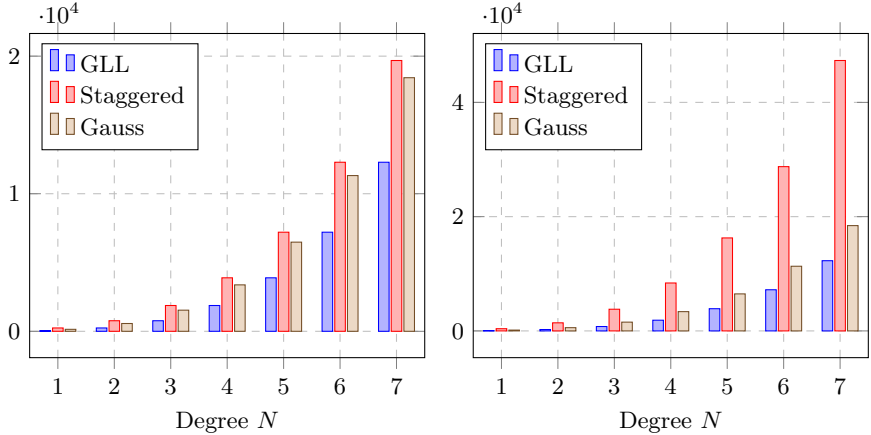
\begin{figure}
\centering
\subfloat[Two-point flux evaluations per dimension]{
\begin{tikzpicture}
\begin{axis}[
	width=.525\textwidth,
	legend cell align=left,
	xlabel={Degree $N$},
	xmin=.5, xmax=7.5,
             ybar=2*\pgflinewidth,
	bar width=4pt,
	xtick={1,2,3,4,5,6,7},
	legend pos=north west,
	xmajorgrids=true,
	ymajorgrids=true,
	grid style=dashed,
] 
\addplot coordinates{(1,48)(2,243)(3,768)(4,1875)(5,3888)(6,7203)(7,12288)};
\addplot coordinates{(1,243)(2,768)(3,1875)(4,3888)(5,7203)(6,12288)(7,19683)};
\addplot coordinates{(1,144)(2,567)(3,1536)(4,3375)(5,6480)(6,11319)(7,18432)};
\legend{GLL, Staggered, Gauss}
\end{axis}
\end{tikzpicture}
}
\subfloat[Est.\ matrix computation operations ]{
\begin{tikzpicture}
\begin{axis}[
	width=.525\textwidth,
	legend cell align=left,
	xlabel={Degree $N$},
	xmin=.5, xmax=7.5,
             ybar=2*\pgflinewidth,
	bar width=4pt,             
	xtick={1,2,3,4,5,6,7},
	legend pos=north west,
	xmajorgrids=true,
	ymajorgrids=true,
	grid style=dashed,
] 
\addplot coordinates{(1,48)(2,243)(3,768)(4,1875)(5,3888)(6,7203)(7,12288)};
\addplot coordinates{(1,387)(2,1416)(3,3795)(4,8388)(5,16275)(6,28752)(7,47331)};
\addplot coordinates{(1,144)(2,567)(3,1536)(4,3375)(5,6480)(6,11319)(7,18432)};

\legend{GLL, Staggered, Gauss}

\end{axis}
\end{tikzpicture}}
\caption{Comparison of GLL, staggered grid, and Gauss collocation schemes in terms of number of two-point flux evaluations and operations associated with matrix computations in three dimensions. }
\label{fig:cost}
\end{figure}

Figure~\ref{fig:cost} shows the estimated number of two-point flux evaluations and matrix operations for GLL, staggered grid, and Gauss collocation schemes in three dimensions.  We observe that a straightforward implementation of Gauss collocation does not significantly reduce the number of flux evaluations compared to staggered grid schemes, though Gauss collocation schemes result in a  significantly smaller number of operations from matrix computations.  

We note that the number of two-point flux evaluations and matrix operations will impact runtime differently depending on the implementation and computational architecture.  For example, while flux evaluations typically dominate runtimes for serial CPU implementations at all orders of approximation, they do not contribute significantly to runtimes at polynomial degrees $N=1,\ldots, 8$ for implementations on Graphics Processing Units (GPUs) \cite{wintermeyer2018entropy}.  

\section{Conclusion}

This work shows how to construct efficient entropy stable high order Gauss collocation DG schemes on quadrilateral and hexahedral meshes.  Key to the construction of efficient methods are decoupled SBP operators, which deliver entropy stability and high order accuracy while maintaining compact inter-element coupling terms.  These operators are also compatible with existing entropy stable methods for applying interface dissipation \cite{winters2017uniquely} or imposing boundary conditions.  Numerical experiments demonstrate both the stability and high order accuracy of the proposed Gauss collocation schemes on both affine and curvilinear meshes.  \note{We note that, while the numerical experiments presented here consider only mapped Cartesian domains, the method is also applicable to complex geometries, and future work will focus on studying the performance of such methods on curvilinear quadrilateral and hexahedral unstructured meshes.}

We note that results for Gauss collocation are similar to those attained by entropy stable staggered-grid schemes \cite{parsani2016entropy}, and require a similar number of two-point flux evaluations.  However, Gauss collocation schemes result in a lower number floating point operations from matrix computations compared to staggered-grid methods.  Finally, while a rigorous computational comparison between GLL and Gauss collocation schemes remains to be done, Gauss collocation schemes show significant improvements in accuracy compared to GLL collocation schemes on non-Cartesian meshes.  In particular, for sufficiently warped curvilinear mappings, degree $N$ Gauss collocation schemes achieve an accuracy comparable to degree $(N+1)$ GLL collocation schemes in two and three dimensions.  


\section{Acknowledgments}

Jesse Chan is supported by NSF DMS-1719818 and DMS-1712639.  The authors thank Florian Hindenlang and Jeremy Kozdon for helpful discussions and suggestions.

\appendix

\section{Decoupled SBP operators for general choices of quadrature and basis}
\label{app:decoupled}

For general choices of quadrature and basis, decoupled projection operators involve a volume quadrature interpolation matrix $\bm{V}_q$, a face quadrature interpolation matrix $\bm{V}_f$, and a quadrature-based $L^2$ projection matrix $\bm{P}_q$.  Let $\LRc{\phi_j}_{j=1}^{N_p}$ denote a set of $N_p$ basis functions, and let $\LRc{\bm{x}_i, \bm{w}_i}_{i = 1}^{N_q}$ denote a set of $N_q$ volume quadrature points and weights in $d$ dimensions.  We also introduce the set of $N^f_q$ surface quadrature points and weights $\LRc{\bm{x}^f_i, \bm{w}^f_i}_{i=1}^{N^f_q}$.  
Then, $\bm{V}_q, \bm{V}_f$ are given as 
\begin{align*}
\LRp{\bm{V}_q}_{ij} &= \phi_j(\bm{x}_i), \qquad 1 \leq i \leq N_q, \qquad 1 \leq j \leq N_p,\\
\LRp{\bm{V}_f}_{ij} &= \phi_j\LRp{\bm{x}^f_i}, \qquad 1 \leq i \leq N^f_q, \qquad 1\leq j \leq N_p.
\end{align*}
These matrices can be used to define the quadrature-based $L^2$ projection matrix $\bm{P}_q$.  Let $\bm{W}, \bm{W}_f$ denote the diagonal matrix of volume and surface quadrature weights, respectively.  Then, 
\[
\bm{M} = \bm{V}_q^T\bm{W}\bm{V}_q, \qquad \bm{P}_q = \bm{M}^{-1}\bm{V}_q^T\bm{W}.  
\]
Let $\bm{D}^i$ now denote a modal differentiation matrix with respect to the $i$th coordinate, which maps coefficients in the basis $\phi_j$ to coefficients of the $i$th derivative.  By composing this matrix with interpolation and projection matrices, one can define differencing operators $\bm{D}_q^i = \bm{V}_q\bm{D}^i\bm{P}_q$ which map values at quadrature points to values of approximate derivatives at quadrature points.  Moreover, $\bm{Q}^i = \bm{W}\bm{D}_q^i$ satisfies a generalized SBP property involving the face interpolation and projection matrices $\bm{V}_f, \bm{P}_q$ \cite{chan2017discretely}.  

The decoupled SBP operator $\bm{Q}_N^i$ is then given as
\begin{equation}
\bm{Q}_N^i = \begin{bmatrix}
\bm{Q}^i - \frac{1}{2}\LRp{\bm{V}_f\bm{P}_q}^T\bm{W}_f \diag{\bm{n}_i}{\bm{V}_f\bm{P}_q} & \frac{1}{2}\LRp{\bm{V}_f\bm{P}_q}^T\bm{W}_f \diag{\bm{n}_i}\\
-\frac{1}{2}\bm{W}_f \diag{\bm{n}_i}{\bm{V}_f\bm{P}_q} & \frac{1}{2}\bm{W}_f \diag{\bm{n}_i}
\end{bmatrix}.
\label{eq:qni_quad}
\end{equation}
A straightforward computation shows that $\bm{Q}_N^i$ satisfies an SBP property \cite{chan2017discretely}.  
It is worth noting that the form of $\bm{Q}_N^i$ does not depend on the choice of basis.  So long as the approximation space spanned by the basis $\phi_j$ does not change, the domain and range of $\bm{Q}_N^i$ depend solely on the choice of volume and surface quadrature points.  

A collocation scheme assumes that the number of quadrature points is identical to the number of basis functions.  If the solution is represented using degree $N$ Lagrange polynomials at quadrature points, the matrices $\bm{V}_q, \bm{P}_q$ simplify to
\[
\LRp{\bm{V}_q}_{ij} = \delta_{ij}, \qquad \bm{M} = \bm{W}, \qquad \bm{P}_q = \bm{M}^{-1}\bm{V}_q^T\bm{W} = \bm{I}.
\]
Plugging these simplifications into (\ref{eq:qni_quad}) and restricting to one spatial dimension recovers the decoupled SBP operator (\ref{eq:qndef}).


\bibliographystyle{unsrt}
\bibliography{dg}

\end{document}